\documentclass[11pt]{amsart}

\usepackage{amssymb,amsthm,amsmath,amsrefs, enumerate}
\usepackage{graphicx}

\usepackage[margin=3cm]{geometry}
\usepackage[all]{xy}

\usepackage{hyperref}

\newtheorem{theorem}{Theorem}[section]
\newtheorem{corollary}[theorem]{Corollary}
\newtheorem{lemma}[theorem]{Lemma}
\newtheorem{proposition}[theorem]{Proposition}

\theoremstyle{remark}
\newtheorem{remark}[theorem]{Remark}
\newtheorem*{remark*}{Remark}
\newtheorem{example}[theorem]{Example}
\newtheorem{question}[theorem]{Question}

\numberwithin{equation}{section}

\newcommand{\N}{\mathbb{N}}
\newcommand{\Z}{\mathbb{Z}}
\newcommand{\ZZ}{\mathcal{Z}}
\newcommand{\K}{\mathcal{K}}

\newcommand{\C}{\mathbb{C}}
\newcommand{\id}{\mathrm{id}}

\newcommand{\Cu}{\mathrm{Cu}}

\newcommand{\her}{\mathrm{her}}

\renewcommand{\epsilon}{\varepsilon}
\renewcommand{\leq}{\leqslant}
\renewcommand{\geq}{\geqslant}

\newcommand{\e}{\epsilon}
\newcommand{\dl}{\delta}

\newcommand{\demph}{\textbf}
\newcommand{\F}{\mathrm{F}}

\title{Nuclear dimension and $\ZZ$-stability of non-simple $\mathrm{C}^*$-algebras}

\author{Leonel Robert}
\address{\hskip-\parindent
Department of Mathematics
University of Louisiana at Lafayette
Lafayette, USA.}
\email{lrobert@louisiana.edu}
\author{Aaron Tikuisis}
\address{\hskip-\parindent
Aaron Tikuisis, Institute of Mathematics, University of Aberdeen, Aberdeen, United Kingdom.}
\email{a.tikuisis@abdn.ac.uk}

\subjclass[2010]{46L35 (46L80, 46L05, 46L06, 47L40, 46L85, 46L55)}
\thanks{The second-named author was partially supported by DFG (SFB 878).}

\begin{document}
\maketitle

\begin{abstract}
We investigate the interplay of the following regularity properties for non-simple $\mathrm C^*$-algebras: finite nuclear dimension, $\mathcal Z$-stability, and algebraic regularity in the Cuntz semigroup.
We show that finite nuclear dimension implies algebraic regularity in the Cuntz semigroup, provided that known type I obstructions are avoided.
We demonstrate how finite nuclear dimension can be used to study the structure of the central sequence algebra, by factorizing the identity map on the central sequence algebra, in a manner resembling the factorization arising in the definition of nuclear dimension.

Results about the central sequence algebra are used to attack the conjecture that finite nuclear dimension implies $\ZZ$-stability, for  sufficiently non-type I, separable $\mathrm C^*$-algebras.
We prove this conjecture in the following cases: (i) the $\mathrm C^*$-algebra has no purely infinite subquotients and its primitive ideal space has a basis of compact open sets,
(ii) the $\mathrm C^*$-algebra has no purely infinite quotients and its primitive ideal space is Hausdorff.
In particular, this covers $\mathrm C^*$-algebras with finite decomposition rank and real rank zero.
Our results hold more generally for $\mathrm C^*$-algebras with locally finite nuclear dimension which are $(M,N)$-pure (a regularity condition of the Cuntz semigroup).
\end{abstract}

\section{Introduction}

Many recent advances in the study and classification of nuclear $\mathrm C^*$-algebras have centered around understanding low-dimensional behaviour or regularity.
Examples of  R\o rdam \cite{Rordam:InfiniteAndFinite} and Toms \cite{Toms:annals},
relying on techniques pioneered by Villadsen \cite{Villadsen:Perforation},  demonstrated that some sort of regularity condition, stronger than nuclearity, is necessary in order to have a classification by K-theory and traces.
Three candidate regularity conditions, involving quite diverse ideas, have emerged: finite nuclear dimension, tensorial absorption of the Jiang-Su algebra $\mathcal Z$, and algebraic regularity in the Cuntz semigroup (more explicitly, almost unperforation and almost divisibility); see \cite{ElliottToms} for an overview.
A set of conjectures have been set forth relating these properties for simple, nuclear $\mathrm C^*$-algebras, and a great deal of progress has been made in proving them (see \cite{MatuiSato:Comp,MatuiSato:dr,Jacelon:fdtraces,KirchbergRordam:CentralSeq,Nawata:Projless,Sato:fdTraces,SWW:Znucdim,UnitlessZ,TW:Zdr,TWW:fdTraces,Winter:drZstable,Winter:pure}).

The purpose of this paper is to initiate the study of these regularity properties in the non-simple case.
We first investigate how finite nuclear dimension implies regularity in the Cuntz semigroup.
We show that an algebra of finite nuclear dimension without simple purely infinite subquotients has strong tracial $M$-comparison and $N$-almost divisibility, where $M$ and $N$ are constants (depending on the nuclear dimension); see Theorem \ref{thm:purenonelementary}.
Here and throughout, we use subquotient to mean an ideal of a quotient.
As an application, we give a new proof of a significant result of Dadarlat and Toms, on the $\ZZ$-stability of infinite tensor products \cite{DadarlatToms:InfTens}.
Another application is that, if $A$ is unital, has finite nuclear dimension and no simple purely infinite subquotients, then $W(A)$ is hereditary inside $\Cu(A)$.

We then investigate Kirchberg's central sequence algebra $\F(A)$.
When $A$ is unital, this is the relative commutant of $A$ inside the algebra of bounded sequences in $A$, modulo the $\|\cdot\|$-null sequences relative to some fixed ultrafilter $\omega$.
It is the $\mathrm C^*$-algebraic analogue of the central sequence algebra classically used by McDuff \cite{McDuff:Central}, and has proven to be a powerful tool in understanding $\mathrm C^*$-algebras.
We prove a factorization result for this algebra, allowing certain properties of the algebra $A$, and especially of the sequence algebra $A_\omega$ to be lifted to properties of $\F(A)$; in subsequent work, we show that, in most respects, $\F(A)$ is substantially different from $A_\omega$ except when $A$ is quite special \cite{FRT:CentralSeq}.
Here is a simplified version of the factorization result (important, yet technical, refinements, can be found in the full statement, Theorem \ref{thm:centralfactorization}):

\begin{theorem}
\label{thm:centralfactorizationsimple}
Let $A$ be a $\mathrm C^*$-algebra of nuclear dimension $m$.
Then, on any separable $\mathrm C^*$-subalgebra of $\F(A)$, the identity map can be factorized
\[
\xymatrix{
\F(A) \ar[rr]^-{\bigoplus_k Q_k} && C_0 \oplus \cdots \oplus C_{2m+1} \ar[rr]^-{\sum_k R_k} && \F(A),
}
\]
where $C_k$ is a hereditary subalgebra of $A_\omega$ and $Q_k,R_k$ are c.p.c.\ maps such that $Q_k$ is order zero for all $k$. Moreover, given separable $\mathrm C^*$-subalgebras   $C_k'\subset C_k$ for each $k$, the map $R_k$ may be modified, if necessary, so as to be an order zero map on $C_k'$.
\end{theorem}

This theorem becomes our primary tool for investigating the structure of $\F(A)$ of a $\mathrm C^*$-algebra of finite (or even locally finite) nuclear dimension.
With this tool, we attack the following conjectures:
\begin{enumerate}
\item[(C1)]
If a separable $\mathrm{C}^*$-algebra has finite nuclear dimension and no elementary subquotients then it is $\ZZ$-stable.
\item[(C2)]
If a separable $\mathrm{C}^*$-algebra has locally finite nuclear dimension and is $(M,N)$-pure for some $M,N>0$ then it is $\ZZ$-stable.
\end{enumerate}
In (C2), $(M,N)$-pureness means that the Cuntz semigroup has $M$-comparison and $N$-almost-divisibility.
Notice that the first conjecture follows from the second, provided that one shows that $\mathrm{C}^*$-algebras of finite nuclear dimension and without elementary subquotients are $(M,N)$-pure for some $M,N>0$.

We reduce these to a question of finding full orthogonal elements, see Theorem \ref{thm:OrthogZ}; in particular, we prove:

\begin{theorem}
\label{thm:OrthogZsimple}
Let $A$ be a separable $\mathrm C^*$-algebra with finite nuclear dimension.
Then $A$ is $\ZZ$-stable if and only if there exist two full orthogonal elements in $\F(A)$.
\end{theorem}

Then, in Theorems \ref{thm:AlgLattZ} and \ref{thm:HausdorffZ}, we settle the conjectures (C1) and (C2) under additional assumptions: 
\begin{enumerate}
\item[(A1)]
no simple subquotient of the $\mathrm{C}^*$-algebra is purely infinite, and
\item[(A2)]
the primitive spectrum of the $\mathrm{C}^*$-algebra satisfies either one of the following
\begin{enumerate}
\item it has a basis of compact open sets, or
\item it is Hausdorff.
\end{enumerate}
\end{enumerate}
Every $\mathrm{C}^*$-algebra of locally finite decomposition rank and with the ideal property (i.e., such
that every closed two-sided ideal is generated by its projections) satisfies (A1) and (A2)(a).
Crossed products of the Cantor set by a free $\Z^n$ action also satisfy (A1) and (A2)(a), and have been shown by Szab\'o to have finite nuclear dimension \cite{Szabo:Rokhdim}.
The case (A2)(b) complements the main result of \cite{TW:Zdr}, and can be used to understand the range of possibilities of $C(X)$-algebras with strongly self-absorbing fibres, such as the examples in \cite{HirshbergRordamWinter}.

Our techniques also clarify the simple case: A more conceptual and less involved proof of $\ZZ$-stability in the simple finite case is presented in Section \ref{sec:SimpleZ}.
This section also contains a separate argument for the simple infinite case, that does not appeal to Kirchberg's $\mathcal O_\infty$-absorption theorem.

The obstacle to the complete resolution of the conjectures above is the construction of full orthogonal elements in the central sequence algebra, as shown by Theorem \ref{thm:OrthogZsimple}.
Our approach to constructing full orthogonal elements makes use of the finiteness conditions  (A1) and (A2), which ensure that certain orthogonal elements obtained using Kirchberg's  covering number and functional calculus  are indeed full (see Lemmas \ref{lem:fullorthogonal}, \ref{lem:TwoOrthogTrFull}, and \ref{lem:AlgLattOrthog}).
Example \ref{ex:FullOrthogFail} shows definitively that this construction cannot work without some kind of finiteness condition.
The same construction was used by Winter for the simple case in \cite{Winter:pure}, so that finiteness (or the existence of a nontrivial trace) also underpins the arguments there.

Even the following, much weaker question remains open:
\begin{question}
If $A$ is of finite nuclear dimension and without elementary quotients, does $A$ contain two (almost) full orthogonal elements?
\end{question}
What is meant by $A$ having two \textit{almost} full orthogonal elements is that, given any element $a$ of the Pedersen ideal of $A$, there exist two orthogonal elements, both of which generate an ideal containing $a$.
(It is equivalent to having two full orthogonal elements when $\mathrm{Prim}(A)$ is compact.)

This paper is organized as follows: In Section \ref{sec:Prelim} we cover, among other preliminary facts,
algebraic regularity properties of the Cuntz semigroup, the notion of nuclear dimension, and a criterion for $\ZZ$-stability involving the central sequence algebra.
In Section \ref{sec:Without} we investigate the divisibility properties of $\mathrm{C}^*$-algebras of finite nuclear dimension. We apply these results to give a simple  proof of Dadarlat and Toms's result on the $\ZZ$-stability of infinite tensor products \cite{DadarlatToms:InfTens}.
In Section \ref{sec:Central} we prove the above mentioned factorization of the identity on 
central sequence algebras.
In Sections \ref{sec:Comp} and \ref{sec:Divis} we apply this factorization to investigate comparison and divisibility properties of central sequence algebras.
Finally, Section \ref{sec:ZStability} contains the proofs of $\ZZ$-stability.

\section{Preliminaries}
\label{sec:Prelim}
Let us start by fixing some of the notation that will be used throughout the paper.
Let $A$ be a $\mathrm{C}^*$-algebra.
We denote by $A_+$ the cone of positive elements of $A$ 
and by $A^\sim$ the unitization of $A$.
By a \demph{subquotient} of $A$, we mean an ideal of a quotient.
Let $a\in A_+$. The hereditary subalgebra $\overline{aAa}$ 
will be denoted by $\mathrm{her}(a)$. If $\epsilon>0$ then $(a-\epsilon)_+$ denotes the element obtained by functional calculus evaluating the function $(t-\epsilon)_+:=\max(t-\epsilon,0)$, with $t\geq 0$, on the positive element  $a$. 
We will also frequently use functional calculus with the function  $g_{\epsilon}\in C_0(0,1]$ which is  0 on $[0,\frac{\epsilon}{2}]$,  $1$ on $[\epsilon,1]$ and linear otherwise.
\begin{figure}[htp]
\includegraphics[height=1.2in]{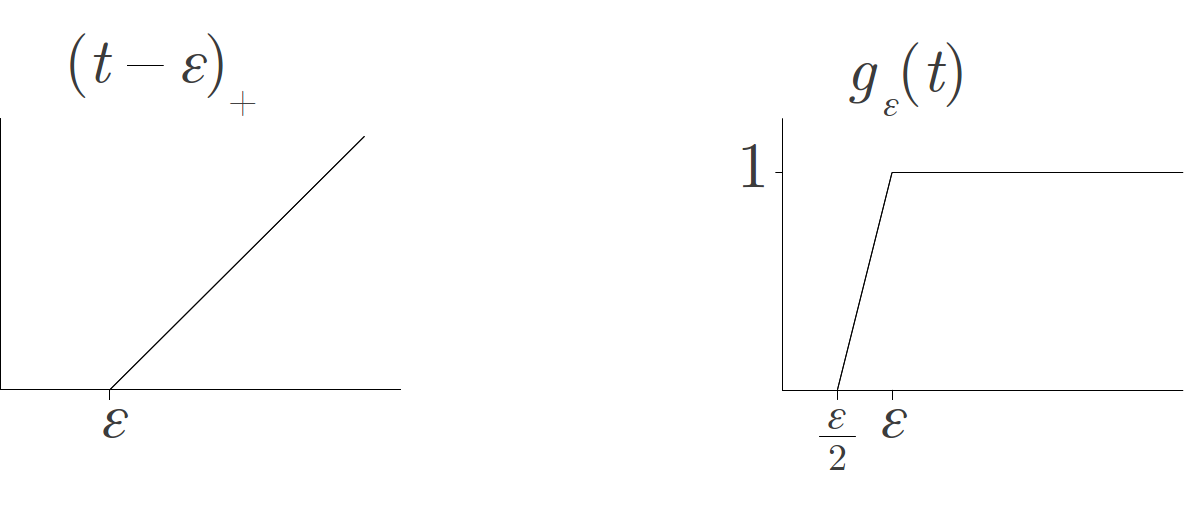}
\caption{Graphs of $(t-\e)_+$ and $g_\e(t)$}
\end{figure}

Let $a,b\in A$. Let us write $a\approx_\epsilon b$ to indicate that $\|a-b\|<\epsilon$. 
The commutator $ab-ba$ is denoted by $[a,b]$. 
If $\alpha\colon A\to B$ is a linear map between $\mathrm{C}^*$-algebras and $b\in B$ then $\|[\alpha,b]\|<\epsilon$ means that
$\|[\alpha(v),b]\|<\epsilon$ for all contractions  $v\in A$. If $\beta\colon A\to B$ is another map then $\|[\alpha,\beta]\|<\epsilon$ means that $\|[\alpha(v),\beta(w)]\|<\epsilon$ for all contractions  $v,w\in A$.

A linear map $\tau\colon A_+\to [0,\infty]$ is called a trace on $A$ if $\tau(0)=0$ and $\tau(x^*x)=\tau(xx^*)$
for all $x\in A$. The cone of lower semicontinuous traces on $A$ is denoted 
by $\mathrm{T}(A)$. 
We emphasize that $\mathrm{T}(A)$ does not denote only the set of bounded traces (or even tracial states), even though that convention has often been used in the literature.
Lower semicontinuous traces on $A$ extended uniquely to lower semicontinuous traces on $A\otimes \mathcal K$. Thus, we will  assume tacitly that the domain of the  traces in 
 $\mathrm{T}(A)$ is $(A\otimes \mathcal K)_+$. Here, and throughout the paper, $\mathcal K$
denotes the $\mathrm{C}^*$-algebra of compact operators on a separable Hilbert space.

\subsection{The Cuntz semigroup}
We will make frequent use of the arithmetic of Cuntz classes of positive elements.
Let us recall the definition of the Cuntz semigroup. Let $A$ be a $\mathrm{C}^*$-algebra.
Let $a,b\in A_+$. Then $a$ is said to be Cuntz smaller than $b$, denoted by $a\precsim b$,  if there exist $d_n\in A$
such that $d_n^*bd_n\to a$; $a$ and $b$ are Cuntz equivalent, denoted by $a\sim b$, if $a\precsim b$ and $b\precsim a$.  The relation $\precsim$ is a pre-order relation and, consequently, $\sim$ is an equivalence relation.

The Cuntz semigroup of the $\mathrm{C}^*$-algebra $A$ is defined as the set of Cuntz equivalence classes of positive element of $A\otimes \mathcal K$. If $a\in (A\otimes\mathcal K)_+$, the Cuntz class of $a$ is denoted by $[a]$. The relation
$[a]\leq [b]$ if $a\precsim [b]$ defines an order on $\Cu(A)$. The addition operation
on $\Cu(A)$ is such that $[a]+[b]=[a'+b']$, where $a\sim a'$, $b\sim b'$, and 
$a'b'=0$ (such elements can always be found using the stability of $A\otimes \mathcal K$).

A positive element $a\in (A\otimes\mathcal K)_+$, and its Cuntz class $[a]$, are called \demph{properly infinite}
if $a \neq 0$ and $2[a] \leq [a]$ in $\Cu(A)$. They are called \demph{stably properly infinite} if $a \neq 0$ and for some
$n\in\N$, $(n+1)[a]\leq n[a]$ (equivalently, for some $n \in \N$, $n[a]$ is properly infinite). 

Let $\tau\in \mathrm{T}(A)$ be a lower semicontinuous trace on $A$. For each $a\in (A\otimes \mathcal K)_+$ let us define 
\[
d_\tau(a)=\lim_n \tau(a^{\frac 1 n}).
\]
The number  $d_\tau(a)$ depends only on the Cuntz class of $a$ and is understood as giving  rise to an additive, order preserving, and supremum preserving map on $\Cu(A)$ (a.k.a., a functional on $\Cu(A)$) given by $[a]\mapsto d_\tau(a)$.
This holds more generally when $\tau$ is a lower semicontinuous 2-quasitrace on $A$ (see \cite[Section II]{BlackadarHandelman}).
A theorem of Haagerup says that if $A$ is an exact $\mathrm{C}^*$-algebra (in particular, if it is nuclear), then a lower semicontinuous 2-quasitrace on $A$ is a trace. However, 
we will often use 2-quasitraces instead of traces in order to state our results in more generality. The cone of lower semicontinuous 2-quasitraces on $A$ will be denoted by $\mathrm{QT}(A)$, and when we will simply say ``quasitrace'' to mean lower semicontinuous 2-quasitrace.

Let $[a],[b] \in \Cu(A)$ and $\gamma > 0$.
We write $[a] \propto [b]$ to mean that $[a] \leq n[b]$ for some $n \in \N$. We write $[a] <_s \gamma[b]$ to mean that there exists $\gamma' < \gamma$ such that $d_\tau(a) \leq \gamma'd_\tau(b)$ for all $\tau \in \mathrm{QT}(A)$.
In the case $\gamma=1$, the relation $<_s$ has been defined elsewhere in the literature with a  slightly different meaning; see for example \cite[Definition 2.2]{OrtegaPereraRordam}.

\subsection{The central sequence algebra}
Let $(A_k)_{k=1}^\infty$ be a sequence of $\mathrm{C}^*$-algebras. Let us denote by $\prod_{k=1}^\infty A_k$ the $\mathrm{C}^*$-algebra of norm-bounded sequences $(a_k)_{k=1}^\infty$
with $a_k\in A_k$ for all $k$. Let $\omega$ be a free ultrafilter in $\N$. Let us denote by
$c_\omega((A_{k})_{k=1}^\infty)$ the closed two-sided ideal of $\prod_{k=1}^\infty A_k$ of sequences $(a_k)_{k=1}^\infty$ for which $\lim_{\omega}\|a_k\|=0$. The \demph{ultraproduct} of the $\mathrm{C}^*$-algebras $A_k$,   $k=1,2,\dots$, is defined as 
\[
\prod_\omega A_k:=\Big(\prod_{k=1}^\infty A_k\Big)/c_\omega((A_{k})_{k=1}^\infty).
\]
Whenever it is clear by the context, we will denote the quotient map from $\prod_{k=1}^\infty A_k$ to $\prod_\omega A_k$ by $\pi_\omega$.
 If $A_k=A$ for all $k=1,\dots$ we denote the ultraproduct by  $A_\omega$  and call it the \demph{ultrapower} of $A$.

Observe that $A$ embeds inside $A_\omega$ as the set of constant sequences.  Let us denote by $A'\cap A_\omega$ the commutant of $A$ inside $A_\omega$, i.e., the elements of $a\in A_\omega$ such that $[a,c]=0$ for all $c\in A$. Let us denote by $A^\perp\cap A_\omega$ (or sometimes simply $A^\perp$) the elements of $A_\omega$ that are orthogonal to $A$, i.e., the elements $a\in A_\omega$ such that $ac=ca=0$ for all $c\in A$. Observe that $A^\perp$ is a a closed two-sided ideal of $A'\cap A_\omega$. The central sequence $\mathrm{C}^*$-algebra is defined as
\[
\F(A):=(A'\cap A_\omega)/A^\perp.
\]
We will also consider the following more general central sequence algebras (studied by Kirchberg in \cite{Kirchberg:CentralSequences}): let $B\subseteq A_\omega$ be a $\mathrm{C}^*$-subalgebra. Let us denote by $B'\cap A_\omega$  its commutant and by $B^\perp$ the subalgebra of $A_\omega$ of elements orthogonal to $B$. The algebra $B^\perp$ is again an ideal of $B'\cap A_\omega$. We define
\[
\F(B,A):=(B'\cap A_\omega)/B^\perp.
\]

\subsection{Divisibility and comparison}
Algebraic regularity properties -- of comparison and divisibility -- in the Cuntz semigroup
of a $\mathrm{C}^*$-algebra play a key role in our arguments. Here we recall $M$-comparison and $N$-almost divisibility, which together form the notion of $(M,N)$-pureness.

Let $M\in\N$.
Let us say that $A$ has \demph{$M$-comparison} if for all $[a],[b_0],[b_1],\dots,[b_M]\in \Cu(A)$ we have that $[a] <_s [b_i]$ for $i=0,\dots,m$ implies that $[a] \leq \sum_{i=0}^M [b_i]$.

Let $N\in \N$.
Let us say that $A$ is \demph{$N$-almost divisible} if for each $[a]\in \Cu(A)$, $k\in \N$ and $\epsilon>0$, there exists $[b]\in \Cu(A)$ such that
\[
k\cdot [b]\leq [a] \hbox{ and }[(b-\epsilon)_+]\leq (k+1)(N+1)[b].
\]

Following Winter \cite{Winter:pure}, we call a $\mathrm{C}^*$-algebra  \demph{$(M,N)$-pure} if it has $M$-comparison and is $N$-almost divisible. (We point out, however, that our definition of $N$-almost divisibility does not exactly agree with Winter's.) 
The comparison and divisibility properties on $\Cu(A)$ relate to $\mathcal Z$-stability and nuclear dimension: If $A$ has nuclear dimension $m$, then it has $m$-comparison \cite{Robert:dimNucComp}, while if $A$ is $\ZZ$-stable then it is $(0,0)$-pure \cite[Proposition 3.7]{Winter:pure}
(cf.\ also Conjecture (C2) and the remarks following it, above).

\begin{lemma}
\label{lem:DivCharacterization}
The following are equivalent.
\begin{enumerate}[(i)]
\item $A$ is $N$-almost divisible.
\item For every $e,a \in A_+, k \in \N$ and $\e > 0$, there exist $v \in M_{(k+1)(N+1) \times 1}(A)$ and a c.p.c.\ order zero map $\phi\colon M_k(\C) \to A$ such that $e\phi(\cdot)=\phi$, $(a-\e)_+=v^*v$ and $v=(\phi(e_{11}) \otimes 1_{(k+1)(N+1)})v$.
\end{enumerate}
\end{lemma}

\begin{proof}
(i) $\Rightarrow$ (ii):
Since $A$ is $N$-almost divisible, there exists $b \in A_+$ such that $k[b] \leq [a]$ and $[(a-\frac\e2)_+] \leq (k+1)(N+1)[b]$.
Let $\dl > 0$ and $x \in M_{(k+1)(N+1) \times 1}(A)$ be such that $(a-\e)_+ = x^*((b-\dl)_+ \otimes 1_{(k+1)(N+1)})x$.
By \cite[Lemma 2.4]{RobertRordam} and \cite[Proposition 2.4]{Rordam:UHFII}, there exist a c.p.c.\ order zero map $\tilde\phi\colon M_k(\C) \to \her(a)$, $\eta > 0$, and $y \in A$ such that $(b-\dl)_+= y^*(\tilde\phi(e_{11})-\eta)_+y$.
Setting $\phi:= g_\eta(\tilde\phi)$ and $v:= ((\tilde\phi(e_{11})-\eta)_+^{1/2}y \otimes 1_{(k+1)(N+1)})x$ we easily see that the properties in (ii) hold.

(ii) $\Rightarrow$ (i):
Let us apply (ii) to $(a-\frac\e2)_+$ in place of $a$, $g_{0,\frac\e2}(a)$ in place of $e$, and $\frac\e2$ in place of $\e$.
Then, with the resulting c.p.c.\ order zero map $\phi$, we can see that $b:= \phi(e_{11})$ satisfies $k[b] \leq [a]$ and $[(a-\e)_+] \leq (k+1)(N+1)[b]$.
\end{proof}

\begin{proposition}
\label{prop:DivCompSequence}
The properties of $M$-comparison and $N$-almost divisibility
pass to quotients and products of $\mathrm{C}^*$-algebras (and in particular, they pass to ultraproducts).
More specifically, given $\mathrm{C}^*$-algebras
$A$ and $(A_\lambda)_{\lambda\in \Lambda}$, if they all have either one of these properties then so do $\prod_{\lambda\in \Lambda} A_\lambda$ and $A/I$ for any closed two-sided ideal $I\subseteq A$.
\end{proposition}

\begin{proof}
It is shown in \cite[Lemma 2.3]{Robert:dimNucComp} that property of being unperforated passes to quotients and products. The same proof, with minor modifications, applies to $M$-comparison.
A key fact is that $(\prod_{\lambda} A_\lambda) \otimes \K$ is a hereditary subalgebra of $\prod_\lambda (A_\lambda \otimes \K)$; this is true because
\[ (\prod_\lambda A_\lambda) \otimes M_n(\C) = \prod_\lambda (A_\lambda \otimes M_n(\C)) \]
is a hereditary subalgebra of $\prod_\lambda (A_\lambda \otimes \K)$ for each $n$ (where we are viewing $M_n(\C)$ as a corner of $\K$).

As for $N$-almost-divisibility, 
it is clear that the condition in Lemma \ref{lem:DivCharacterization} (ii) passes to products and quotients (cf.\ \cite[Proposition 8.4]{RobertRordam}, where divisibility of the unit is shown to pass to sequences).
\end{proof}

\subsection{Nuclear dimension}
Let $A$ and $B$ be $\mathrm{C}^*$-algebras.
Let $\phi\colon A\to B$ be a completely positive contractive (c.p.c.) map. Let us say that
$\phi$ has \demph{order zero} if it preserves orthogonality, i.e., $ab=0$ implies $\phi(a)\phi(b)=0$ for all $a,b\in A$. By \cite[Theorem 2.3]{WinterZacharias:Order0}, any such map has the form $\phi(a)=h\pi_\phi(a)$, where $\pi_\phi\colon A\to M(C^*(\phi(A))$ is a homomorphism, and $h\in M(C^*(\phi(A)))$ commutes with $\phi(A)$. We will make use of the functional calculus on order zero maps introduced by Winter and Zacharias: if the function $f\in C_{0}(0,\|\phi\|]$ is positive and of norm at most 1, then we set $f(\phi):=f(h)\pi_\phi$, which is also a c.p.c.\ map of order zero from $A$ to $B$, and it satisfies $f(\phi)(p)=f(\phi(p))$ for every projection $p \in A$.

Following Winter and Zacharias \cite{WinterZacharias:NucDim}, we say that a $\mathrm{C}^*$-algebra $A$ has \demph{nuclear dimension} at most $m$
if for each finite set $F\subset A$ and $\epsilon>0$ there exist c.p.c.\ maps 
\[
\xymatrix{
A\ar[r]^{\psi_k} & C_k\, \ar[r]^{\phi_k} &A\\
}
\]
with $k=0,1,\dots,m$,
such that $\phi_k$ is an order zero map for all $k$ and 
\[
a\approx_\epsilon\sum_{k=0}^m\phi_k\psi_k
\]
for all $a\in F$.

Recall that a $\mathrm{C}^*$-algebra of finite nuclear dimension has the $m$-comparison property. We point out the following consequence of the $m$-comparison property (proven in 
\cite[Theorem 5.4]{WinterZacharias:NucDim} by different means):
\begin{proposition}\label{prop:nucleardichotomy}
If a $\mathrm{C}^*$-algebra $A$ is simple, of finite nuclear dimension, and traceless, then $A$ is purely infinite.
\end{proposition}
\begin{proof}
Let $A$ be traceless and of finite nuclear dimension. By the $m$-comparison property we have that $m[a]$ is properly infinite for any non-zero $[a]\in \Cu(A)$. But since $A$ is simple and non-type I,
by Glimm's Halving Lemma \cite[Lemma 6.7.1]{Pedersen:CstarBook}, we have that for any non-zero $[b]$ there exists a non-zero $[a]$ such that $m[a]\leq [b]$, whence $[b]$ is properly infinite. 
It follows that $A$ is purely infinite. 
\end{proof}

\subsection{The Jiang-Su algebra}
Let us denote by $\ZZ_{k-1,k}$, with $k\in \N$, the prime dimension drop $\mathrm{C}^*$-algebras and by $\ZZ$ the Jiang-Su algebra.  

A $\mathrm{C}^*$-algebra $A$ is called $\ZZ$-stable or tensorially $\ZZ$-absorbing if $A\cong A\otimes \ZZ$. If $A$ is separable, this is equivalent to having a unital embedding of $\ZZ$ in $\F(A)$ (see \cite[Theorem 7.2.2]{Rordam:ClassBook}). In fact,  by \cite[Proposition 5.1]{Nawata:Projless} (cf.\ \cite[Proposition 2.2]{TomsWinter:ZASH}), it suffices to find unital embeddings of the dimension drop $\mathrm{C}^*$-algebras $\ZZ_{k-1,k}$ into $\F(A)$  for all $k\in \N$.
Furthermore, R\o rdam and Winter showed in \cite[Proposition 5.1]{RordamWinter:Z}, that in order to have one such embedding it suffices to find a c.p.c.\ order zero map from $M_{k-1}(\C)$ into $\F(A)$ with ``small defect".
Thus, we arrive at the following $\ZZ$-stability criterion:

\begin{proposition}[cf.\ \cite{Winter:pure}*{Proposition 1.14}]\label{prop:Zcriterion}
Let $A$ be a separable $\mathrm{C}^*$-algebra. Then $A$ is $\ZZ$-stable if and only if for each $k\in \N$ there exists a c.p.c.\ map of order zero  $\phi\colon M_k(\C)\to \F(A)$ such that 
$[1-\phi(1)]\ll [\phi(e_{11})]$ in the Cuntz semigroup of $\F(A)$. 
\end{proposition}

\section{Divisibility for $\mathrm{C}^*$-algebras of finite nuclear dimension}
\label{sec:Without}

In this section, we prove the following:
\begin{theorem}\label{thm:purenonelementary}
Given $m \in \N$, there exists $N \in \N$ such that the following holds:
If $A$ is a $\mathrm{C}^*$-algebra of  nuclear dimension $m$, with no elementary subquotients and no simple purely infinite subquotients, then $A$ is $(m,N)$-pure.
\end{theorem}

That $A$ has $m$-comparison has already been shown, by the first-named author in \cite{Robert:dimNucComp}, so what is really proven here is $N$-almost-divisibility.
This will be deduced from a quantitative analysis of the relation between the size of the finite dimensional representations of $A$ and divisibility properties (in the Cuntz semigroup) of a strictly positive element in $A$.
It is likely that the same result holds after dropping the finiteness condition of no simple purely infinite subquotients, but our present methods -- specifically, the construction of almost full orthogonal elements in Lemma \ref{lem:fullorthogonal} -- require it, as demonstrated in Example \ref{ex:FullOrthogFail}.
Notice that if $A$ has finite decomposition rank then it satisfies this condition, since its simple subquotients 
also have finite decomposition rank and thus cannot be purely infinite.

\begin{proposition}\label{prop:weakdiv}
Let $m,k\in \N$. Let $A$ be a $\mathrm{C}^*$-algebra  of nuclear dimension $m$ and such that
every representation of $A$ has dimension at least $k$.

\begin{enumerate}[(i)]
\item
 For each $\epsilon>0$ and strictly positive $c\in A_+$
there exist c.p.c.\ maps of order zero $\phi^j\colon M_k(\C)\to A$, with $j=1,2,\dots,2(m+1)$, such that 
\[
[(c-\epsilon)_+]\leq \Big[\sum_{j=1}^{2(m+1)} \phi^j(1) \Big].
\]

\item
 For each $\epsilon>0$ and strictly positive $c\in A_+$ there exists $a\in A_+$ such that
\[
[(c-\epsilon)_+]\leq k[a]\leq 2(m+1)[c].
\] 
\end{enumerate}
\end{proposition}
\begin{proof}
(i): Let $A\stackrel{\psi^j}{\longrightarrow}F_j\stackrel{\phi^j}{\longrightarrow}A$, with $j=0,1,\dots,m$, be an approximate factorization of $\mathrm{id}_A$, where the maps $\phi^j$ are c.p.c.\ of order zero, the algebras $F_j$ are finite dimensional, and 
\[
\sum_{j=0}^m \phi^j\psi^j(c)\approx_\epsilon c.
\] 
Then $[(c-\epsilon)_+]\leq \sum_{j=0}^m [\phi^j(1_{F_j})]$.  
If every representation of $A$ has dimension at least $k$, we may assume that the matrix sizes of every matrix summand of each $F_j$ are all at least $k$
(by \cite[Proposition 3.4]{WinterZacharias:NucDim}). 
This implies that for each $j$ there exist c.p.c.\ maps of order zero  $\phi^j_1,\phi^j_2\colon M_k(\C)\to F_j$ such that $[1_{F_j}]\leq [\phi^j_1(1)]+[\phi^j_2(1)]$.
The collection of maps $\phi^j_i$, with $j=0,1,\dots,m$ and $i=0,1$ has the desired properties.

(ii): Simply set $a:=\sum_{j=1}^{2(m+1)}\phi^j(e_{11})$, with $\phi^j\colon M_k(\C)\to A$
as in part (i). The desired properties for $a$ are readily verified.
\end{proof}

\begin{lemma}
\label{lem:nofullcompactpi}
Let $A$ be a $\mathrm{C}^*$-algebra with finite nuclear dimension and  no simple purely infinite quotients.  
Then neither the Cuntz semigroup of $A$ nor of its quotients can contain a full, compact, and properly infinite element.
\end{lemma}

\begin{proof}
Let us argue by contradiction. Stabilizing and passing to a quotient of $A$ if necessary, let us assume that there exists a full element $a\in A_+$ such that $[a]\ll [a]$ and $2[a]=[a]$. Let $\epsilon>0$ be such that $[(a-\epsilon)_+]=[a]$. Since Cuntz equivalent elements generate the same ideal, $(a-\epsilon)_+$ is also full. It follows that there exists at least one  proper maximal ideal $I$ of $A$. Then $A/I$ is a simple $\mathrm{C}^*$-algebra  of finite nuclear dimension containing a compact, stably properly infinite, positive element. By Proposition \ref{prop:nucleardichotomy}, this  $\mathrm{C}^*$-algebra is purely infinite, which contradicts our hypotheses.
\end{proof}

The next lemma deals with the construction of full orthogonal elements.
The construction is essentially the same one pioneered by Winter in \cite[Proposition 3.6]{Winter:drZstable}.

\begin{lemma}\label{lem:fullorthogonal}
Given  $m,l\in \N$  there exist $K,L>0$ with the following property:
If  $A$ is a  $\mathrm{C}^*$-algebra of
nuclear dimension  at most $m$,
such that every representation has dimension at least $K$, and
$A$ has no simple purely infinite quotients,
then for each $\epsilon>0$ and strictly positive element $c\in A_+$ there exist
 mutually orthogonal elements $d^0,d^1,d^2,\dots,d^l\in A_+$  such that 
$[(c-\epsilon)_+]\leq	 L[d^i]$ for  $i=0,1,\dots,l$.
\end{lemma}

\begin{proof}
Let us first deal with the case $l=1$.
Let $A$ be as in the statement. (The values of $K$ and $L$ will be specified in the argument that follows.) Let $\epsilon>0$ and let $c\in A_+$ be strictly positive.
By the Proposition \ref{prop:weakdiv} (ii), if $K\geq 2m+3$ then there exists  $a\in A_+$
such that 
\[
[\Big(c-\frac{\epsilon}{2}\Big)_+]\leq (2m+3)[a]\leq 2(m+1)[c].
\]
Let $\delta>0$ be such that $[(c-\epsilon)_+]\leq (2m+3)[(a-\delta)_+]$. 
Let us define
\begin{align*}
d^0 &=g_\delta(a),\\
d^1 &=(1-g_{\frac \delta 2}(a))^{\frac 1 2}c(1-g_{\frac \delta 2}(a))^{\frac 1 2}.
\end{align*}
It is clear that $d^0$ and $d^1$ are orthogonal and that $[(c-\epsilon)_+]\leq (2m+3)[d^0]$.
As for $d^1$, we have that
\begin{align}\label{fullineq}
[c]\leq [g_{\frac\delta 2}(a)]+[d^1].
\end{align}
Let $\bar\epsilon>0$ be such that $(2m+3)[g_{\frac\delta 2}(a)]\leq (2m+2)[(c-\bar\epsilon)_+]$. Multiplying by $2m+3$ in \eqref{fullineq} we get
\begin{align}\label{cancel}
(2m+3)[c] &\leq (2m+2)[(c-\bar\epsilon)_+]+(2m+3)[d^1].
\end{align}

Let us show that $d^1$ is full in $A$. Let $I$ be the closed two-sided ideal generated
by $d^1$. Passing to the quotient by $I$ in \eqref{cancel} we get  
$(2m+3)[c_I]\leq (2m+2)[(\pi_I(c)-\bar\epsilon)_+]$.
Thus, $(2m+2)[\pi_I(c)]$ is   properly infinite and compact. By the previous lemma, $\pi_I(c)=0$; i.e., $d^1$ is full.

Since $d^1$ is full, a finite multiple of $[d^1]$ majorizes $[g_{\frac\delta 2}(a)]$. Thus, by \eqref{fullineq}, a finite multiple of $[d^1]$ majorizes $[c]$. 
Now from \eqref{cancel}
we deduce that $d_\tau(c)\leq (2m+3)d_\tau(d^1)$ for all $\tau\in \mathrm{T}(A)$. By the $m$-comparison property
this implies that $[c]\leq 2(m+1)(2m+3)[d^1]$. This completes the proof for $l=1$.

For the general case we proceed by induction. From the relation $[c]\leq 2(m+1)(2m+3)[d^1]$ we deduce that if 
all the representations of $A$ have large enough dimension, then so do the representations of $\her(d_1)$ (in a way that depends only
on $m$). Thus, we can apply the induction hypothesis to the hereditary subalgebra generated by $d^1$.
\end{proof}

\begin{example}\label{ex:FullOrthogFail}
The construction of full orthogonal elements in Lemma \ref{lem:fullorthogonal} uses the fact that $c$ from Proposition \ref{prop:weakdiv} has small trace, so that under the right finiteness conditions, $1-g_\e(c)$ is full.
However, if $A$ is simple, unital, and purely infinite, then (for any $k$) there are c.p.c.\ order zero maps $\phi^j$, for $j=1,2$ satisfying (i) of Proposition \ref{prop:weakdiv}, with
\[ \phi^1(e_{11}) + \phi^2(e_{11}) = 1 \]
(it is enough to get these maps into $\mathcal O_2$, which is easy.)
Using such maps, the construction of $c$ in the proof of Proposition \ref{prop:weakdiv} then yields $c=1$, so that there is no way to use functional calculus on $c$ to produce full orthogonal elements.

This demonstrates that an entirely different approach to constructing full orthogonal elements is needed to go beyond situations where finiteness conditions are assumed.
This problem is also present in the argument in \cite{Winter:pure}.
\end{example}

\begin{lemma}\label{lem:largerepsdiv}
Given $m\in \N$ there exist $M,N>0$ with the following property: If $k\in \N$ and $A$ is a $\mathrm{C}^*$-algebra  of nuclear dimension at most $m$, such that every representation has dimension at least $k\cdot M$, and $A$ has
no simple purely infinite quotients,
then for each $\epsilon>0$
and strictly positive $c\in A_+$
there exists $b\in A_+$ such that $k[b]\leq [c]$ and $[(c-\epsilon)_+]\leq kN[b]$.
\end{lemma}

\begin{proof}
Let $K$ and $L$ be constants as in the previous proposition corresponding to $l:=m+1$.
By Proposition \ref{prop:weakdiv} (ii), if every representation of $A$ has dimension at least $2k(m+1)L$, then there
exists $a\in A_+$ such that
\begin{align*}
[(c-\frac \epsilon 2)_+]\leq 2kL(m+1)[a]\leq 2(m+1)[c].
\end{align*}
Let us choose $\delta_1>0$ first, and then $\delta_2>0$, such that
\begin{align*}
[(c-\epsilon)_+]\leq 2kL(m+1)[(a-\delta_1)_+]\leq 2(m+1)[(c-\delta_2)_+].
\end{align*}
If every representation of $A$ has dimension at least $K$, then there exist mutually orthogonal elements $d^0,d^1,\dots,d^{m}\in A_+$ such that
$[(c-\delta_2)_+]\leq L[d^i]$ for all $i$. It follows that
\[
2kL(m+1)[(a-\delta_1)_+]\leq 2L(m+1)[d^i]
\]
for all $i$. Thus, by the $m$-comparison property
\[
k[(a-\delta_1)_+]\leq \sum_{i=0}^m [d^i]\leq [c].
\]
Therefore, setting $b:=(a-\delta_1)_+$, $M:=\max(K,2L(m+1))$, and $N:=2L(m+1)$ (both of which only depend on $m$), we get the desired result.
\end{proof}

\begin{proof}[Proof of Theorem \ref{thm:purenonelementary}]
By \cite{Robert:dimNucComp}, $A$ has $m$-comparison. Let $N>0$ be as in the previous lemma. Since no subquotient of $A$ is elementary, for each 
$a\in (A\otimes \mathcal K)_+$ the $\mathrm{C}^*$-algebra $\mathrm{her}(a)$ has no finite dimensional representations. So 
the previous lemma is applicable to $\mathrm{her}(a)$ and any $k\in\N$, whence showing that $A$ is $N$-almost divisible.
\end{proof}

Let us say that the $\mathrm{C}^*$-algebra $A$ has strong tracial $M$-comparison if for all $[a],[b]\in \Cu(A)$, we have that $[a] <_s \frac1M [b]$ implies that $[a] \leq [b]$.

\begin{theorem}\label{thm:strongtracial}
Let $m\in \N$. There exists $M>0$ such that if $A$ is a $\mathrm{C}^*$-algebra of nuclear dimension at most $m$ with no simple purely infinite  subquotients   then $A$ has strong tracial $M$-comparison.
\end{theorem} 

\begin{proof}
This argument is akin to an argument in \cite[Section 3]{Winter:pure}, that $(M,N)$-pureness implies strong tracial $\overline{M}$-comparison, for some $\overline M$.
However, extra steps are taken here, to avoid assuming that $A$ has no elementary subquotients.

Say $Md_\tau(a)\leq d_\tau(b)$ for all $\tau\in \mathrm{T}(A)$ (how large $M$ should be will be specified later).
Letting $\tau$ be a trace that is 0 on a closed two-sided ideal and $\infty$ outside,  we conclude that the ideal generated by $a$ is contained in the ideal generated by $b$.
We may reduce to the case  that $b$ generates the same ideal as $a$. To see this, let $e_0\in (A\otimes \mathcal K)_+$
be a strictly positive element of the ideal generated by $a$. Let $\overline{b}=e_0be_0$. Then
$[\overline b]\leq [b]$ and $Md_\tau(a)\leq d_\tau(e_0be_0)$ for all $\tau$.

So let us assume that $a$ and $b$ generate the same ideal. We claim that each  representation of $\her(b)$
has dimension at least $M$. Indeed,  no such representation, after being extended to the ideal generated by $b$, can vanish on $a$ (since $a$
is a full element of this ideal). Then $Md_\tau(a)\leq d_\tau(b)$ implies the claim.

Let $\epsilon>0$ and choose $\delta>0$ such that $Md_\tau((a-\epsilon)_+)\leq d_\tau((b-\delta)_+)$ for all $\tau\in \mathrm{T}(A)$. 
If $M$ is large enough (depending only on $m$), there exist -- by Lemma \ref{lem:fullorthogonal} -- mutually orthogonal positive elements $d^0,d^1,\dots,d^m\in \mathrm{her}(b)$ such that $[(b-\delta)_+]\leq L[d^i]$
for all $i$ and some $L>0$. Thus, $Md_\tau((a-\epsilon)_+)\leq L[d^i]$ for all $i$ and $\tau\in \mathrm{T}(A)$. Again, if $M$ is large enough (relative to $L$, which again depends only on $m$), then by 
$m$-comparison we conclude that
\[
[(a-\epsilon)_+]\leq \sum_{i=0}^{m} [d_i]\leq [b].
\]
Since $\epsilon>0$ can be arbitrarily small, we get $[a]\leq [b]$, as desired.
\end{proof}

\subsection{$\ZZ$-stability of infinite tensor products}

In \cite{DadarlatToms:InfTens}, Dadarlat and Toms showed that if  a unital $\mathrm{C}^*$-algebra $A$ admits a unital embedding of
an approximately subhomogeneous $\mathrm{C}^*$-algebra without 1-dimensional representations, then $\bigotimes_{n=1}^\infty A$
is $\ZZ$-stable. As shown in \cite[6.3]{DadarlatToms:InfTens}, this question quickly reduces to the case that $A$ is an RSH algebra
with finite topological dimension and without 1-dimensional representations. The proof in \cite{DadarlatToms:InfTens} then relies
on sophisticated tools from homotopy theory. We give here a more abstract proof of Dadarlat and Toms's result using the results on divisibility previously obtained in this section. 
We prove the following:

\begin{theorem}\label{thm:d-t}
Let $A$ be a separable unital $\mathrm{C}^*$-algebra such that 
\begin{enumerate}
\item
$A$ has no 1-dimensional representations,
\item
$A$ satisfies that
\begin{align}\label{dimgrowth}
\frac{\dim_{\mathrm{nuc}}(A^{\otimes n})}{\alpha^n}\to 0,
\end{align}
for any $\alpha>1$, 
\item
for all $n$, no simple quotient of $A^{\otimes n}$ is purely infinite (e.g., if $A$ has finite decomposition rank). 
\end{enumerate}
Then $A^{\otimes \infty}$ is $\ZZ$-stable. More generally, the same conclusion holds if $A^{\otimes \infty}$
admits a unital embedding of a $\mathrm{C}^*$-algebra with these properties.
\end{theorem}

Conditions (ii) and (iii) above are satisfied if $A$ is an RSH algebra of finite
topological dimension. Indeed, by \cite[Theorem 1.6]{Winter:drSH}, in this case $A$ has finite decompoition rank and $\dim_{\mathrm{nuc}}(A^{\otimes n})$ has linear growth. In this way we recover Dadarlat and Toms's result.

Although we will not use any of the results in this section in the sequel, many of the ideas encountered here will reappear.
A simplification here is that it is easy to arrange commutativity in $A^{\otimes \infty}$.

For the remainder of this section, we let $A$ denote a separable unital $\mathrm{C}^*$-algebra that satisfies (i)-(iii)
of the above theorem.

\begin{lemma}
There exists $k$ such that $A^{\otimes k}$ has two full orthogonal elements.
\end{lemma}
\begin{proof}
By the proof of Lemma \ref{lem:fullorthogonal}, if a unital $\mathrm{C}^*$-algebra $B$ has no simple purely infinite quotients and all its representations have  dimension at least $2\dim_{\mathrm{nuc}}(B)+3$ then $B$ contains two full orthogonal elements. But all the representations of $A^{\otimes k}$ have dimension at least $2^k$, which,  by \eqref{dimgrowth},  majorizes 
$2\dim_{\mathrm{nuc}}(A^{\otimes k})+3$ for $k$ large enough. Thus, the result follows.
\end{proof} 

\begin{lemma}
For all $q\in\N$ there exists $k\in \N$ such that there exists an order zero map $\phi\colon M_q(\C)\to A^{\otimes k}$
whose image is  full.
\end{lemma}
\begin{proof}
We may assume without loss of generality that $q=2^n$ for some $n\in \N$. 
Replacing $A$ by $A^{\otimes k}$, with $k$ as in the previous lemma, we may also assume that $A$ contains two full orthogonal elements.

Let $\gamma_n=\dim_{\mathrm{nuc}}(A^{\otimes n})$. 
Since every representation
of $A^{\otimes n}$ has dimension at least $2^n$, there exist order zero maps $\psi_i\colon M_{2^n}(\C)\to A^{\otimes n}$,
with $i=1,2,\dots,2(\gamma_n+1)$ such that $[1]\leq \sum_{i=1}^{2(\gamma_n+1)} [\psi_i(1)]$ (by Proposition \ref{prop:weakdiv}).
On the other hand, since $A$ contains two full positive orthogonal elements, $A^{\otimes m}$ contains $2^m$ full and pairwise orthogonal positive elements for all $m\in \N$. Let us choose $m$ large enough
such that $2^m\geq 2(\gamma_n+1)$ and let us denote these orthogonal elements by $d_0,d_1,\dots,d_{2^m}\in A^{\otimes m}$. Let us define $\phi\colon M_{2^n}(\C)\to A^{\otimes n}\otimes A^{\otimes m}$ by
\[
\phi=\sum_{i=1}^{2(\gamma_n+1)} \phi_i\otimes d_i.
\]
It can be readily verified that $\phi$ has the desired properties.
\end{proof}

\begin{proof}[Proof of Theorem \ref{thm:d-t}]
By Proposition \ref{prop:Zcriterion}, we must construct for each $q\in \N$ a c.p.c. map of  order zero  $\phi\colon M_q(\C)\to \F(A^{\otimes\infty})$ such that $[1-\phi(1)]\ll [\phi(e_{11})]$. In fact, it suffices to construct one such map $\phi$ from $M_q(\C)$ into $A^{\otimes\infty}$  (by then considering the central sequence of maps 
$\phi\otimes 1\otimes\cdots$, $1\otimes \phi\otimes 1\otimes \cdots$, etc, from
$M_q(\C)$ to $A^{\otimes \infty}\otimes A^{\otimes \infty}\otimes \cdots\cong A^{\otimes\infty}$). Let us do this. 

Let $A$ be a $\mathrm{C}^*$-algebra that satisfies conditions (1)-(3) of the theorem.
By the previous lemma, we may assume that there exists $\psi\colon M_q(\C)\to A$ such that
$\psi(1)$ is full, i.e., $[1]\leq Q[\psi(1)]$, with $Q>0$. Using functional calculus on
the order zero map $\psi$, we may also assume that $2Q[1-\psi(1)]\leq (2Q-1)[1]$ (see the proof of Lemma \ref{lem:DefectManeouvre} below). 
Let $\epsilon>0$ be
such that $[1]\leq Q[(\psi(1)-\epsilon)_+]$.

Let $n\in \N$.
Let $\psi_i\colon M_q(\C)\to A^{\otimes n}$, with $i=1,2,\dots,n$ be given by 
$\psi_i=1\otimes\cdots\otimes \psi\otimes\cdots\otimes 1$. By Lemma \ref{lem:commutingranges} (ii) (essentially, Winter's \cite[Lemma 2.3]{Winter:Zinitial}), there exists a c.p.c. map of order zero $\phi\colon M_q(\C)\to A^{\otimes n}$ such that
$\psi_1\leq \phi$ and
\[
1-\phi(1)=\prod_{i=1}^n (1-\psi_i(1))=\bigotimes_{i=1}^n (1-\psi(1)).
\]
Thus, we find that
\[
(2Q)^n[1-\phi(1)]\leq (2Q-1)^n[1].
\]

Let $\gamma_n=\dim_{\mathrm{nuc}}(A^{\otimes n})$. Let $d^0,d^1\in A_+$ be orthogonal and such that $[1]\leq L[d^i]$ for some $L>0$ and $i=0,1$.
Set $m = \lceil\log_2(\gamma_n+1)\rceil$.
In $A^{\otimes m}$ we can find $2^m$ (approximately $\gamma_n+1$) positive orthogonal elements $d_1,d_2,\dots,d_{2^m}$ such that $[1]\leq L^m[d_i]$
for all $i$.
 Let us choose $m=\lceil\log_2(\gamma_n+1)\rceil$ (so that there are approximately
$\gamma_n+1$ orthogonal elements). Notice that $m<n$ for $n$ lare enough by \eqref{dimgrowth}.
Let us regard $A^{\otimes m}$ as a subalgebra of $1\otimes A^{\otimes n-1}$.
Then,  
\begin{align*}
(2Q)^n[1-\phi(1)] &\leq (2Q-1)^n[1]\\
&\leq (2Q-1)^nQq[\psi_1((e_{11}-\epsilon)_+]\\
&\leq (2Q-1)^nQqL^m[\psi((e_{11}-\epsilon))_+\otimes d_i]
\end{align*}
for all $i=1,\dots,2^m$. We claim that
\begin{equation}
\label{d-tLimit}
\frac{(2Q-1)^nQqL^m}{(2Q)^n}\to 0.
\end{equation}
Indeed, notice that $L^m=L^{\lceil\log_2(\gamma_n+1)\rceil}=O(\gamma_n^{\log_2 L})$, so \eqref{d-tLimit} follows from \eqref{dimgrowth}.
Thus, choosing
$n$ large enough, we get that
\[
[1-\phi(1)]<_s [\psi((e_{11}-\epsilon))_+\otimes d_i]
\]
for all $i=1,2,\dots,2^m$. By the $\gamma_n$-comparison property in $A^{\otimes n}$ \cite{Robert:dimNucComp},
we conclude that
\[
[1-\phi(1)]\leq \sum_{i=1}^{\gamma_n+1}[\psi((e_{11}-\epsilon))_+\otimes d_i]\leq 
[\psi((e_{11}-\epsilon))_+\otimes 1]\ll [\phi(e_{11})].
\]
This completes the proof.
\end{proof}

\subsection{When is $W(A)$ hereditary?}

In the following, we will say that positive elements $a,b \in A_+$ are \textbf{Murray-von Neumann equivalent} if there exists $x \in A$ such that $x^*x = a$ and $xx^* = b$.

\begin{proposition}
Let $A$ be a $\mathrm C^*$-algebra with strong tracial $M$-comparison, such that in every quotient, every full projection is finite.
Let $a \in A_+, b \in (A \otimes \K)_+$, and suppose that $a \in A_+$ is strictly positive and
\[ [b] <_s \frac1{M+1}[a], \]
and that $d_\tau(b)<\infty$ for all densely finite $\tau \in \mathrm{QT}(A)$.
Then $a$ is Murray-von Neumann equivalent to an element of $M_2(A)_+$.
\end{proposition}

\begin{proof}
Let $\gamma < \frac1{M+1}$ be such that $[b] <_s \gamma[a]$.
We follow the proof of \cite[Theorem 4.4.1]{BRTTW}, with modifications as follows.
Note that Cuntz classes are denoted using $\langle\cdot\rangle$ in \cite{BRTTW}, though we will stick with our convention here and use $[\cdot]$.

In the proof of \cite[Theorem 4.4.1]{BRTTW}, we first replace $n+k$ by $1$ throughout.
In place of (7) of \cite{BRTTW}, we use $[b] <_s \gamma[a]$, and then we still get $[h_1(b)] \leq [b] \leq [a]$ (using the $[\cdot]$ notation for Cuntz classes of \cite{BRTTW}).

When we recursively define the $z_i$'s, we replace the part of the inductive hypothesis \cite[Page 3672, 6th line from the bottom]{BRTTW} that reads
\[ \lambda( [h_i(b)] ) + (r_{A,a} + \varepsilon)\lambda([a]) \leq \lambda([y_i])\quad\text{ for all }\lambda \in F(\Cu(A)) \]
with
\[ \lambda( [h_i(b)]) \leq \gamma \lambda([y_i])\quad \text{ for all }\lambda \in F(\Cu(A)), \]
with strict inequality for those $\lambda$ which are densely finite.

We then replace a string of 3 inequalities, \cite[Page 3673, lines 2-4]{BRTTW} by
\begin{align*}
\lambda([h_{i+1}(b)]) + \lambda([g_i(b)]) &\leq \lambda([h_i(b)]) \\
&\leq \gamma \lambda( [y_i] ) \\
&\leq \gamma (\lambda([y_i-s_is_i^*]) + \lambda( [g_i(b)])) \\
&\leq \gamma \lambda([y_i-s_is_i^*]) + \lambda( [g_i(b)]),
\end{align*}
which strict inequality for those $\lambda$ which are densely finite.
If $\lambda([g_i(b)] ) < \infty$ then, in place of (8) in \cite{BRTTW}, we conclude
\[ \lambda( [h_{i+1}(b)]) < \gamma \lambda([y_i - s_is_i^*]). \]
We then use the same argument as in the following paragraph of \cite{BRTTW} to conclude that this holds (with nonstrict inequality) even if $\lambda([g_i(b)] ) = \infty$.

The rest of the proof of \cite[Theorem 4.4.1]{BRTTW} is then unchanged.
\end{proof}

Recall that $W(A)$ denotes the subset of $\Cu(A)$ consisting of elements $[a]$ that are represented by $a \in \bigcup_k M_k(A)_+$.
The question of whether $W(A)$ is a hereditary subset of $\Cu(A)$ has been raised, for example, in \cite{AntoineBosaPerera} and in \cite[Question 4.4.3]{BRTTW}, and the hypothesis that $W(A)$ is hereditary inside $\Cu(A)$ has been used in results of \cite{AntoineBosaPerera,ABPP}.

\begin{corollary}
Let $A$ be a unital $\mathrm C^*$-algebra with finite nuclear dimension and with no simple, purely infinite subquotients.
Then $W(A)$ is a hereditary subset of $\Cu(A)$.
\end{corollary}

\begin{proof}
Since $A$ is unital, we must show that if $[b] \leq n[1]$ in $\Cu(A)$ for some $n \in \N$ then $[b] \in W(A)$.
Since $M_n(A)$ satisfies the hypotheses, we may assume that $n=1$.

We use Lemma \ref{lem:nofullcompactpi} and Theorem \ref{thm:strongtracial} to verify the hypotheses of the previous proposition with some $M \in \N$.
\[ [b] <_s \frac1{M+1} (M+2)[1]. \]
It clearly follows that $d_\tau(b) < \infty$ for any densely finite $\tau$, so that by the previous proposition, $b$ is Murray-von Neumann equivalent to an element of $M_{2(M+2)}(A)$; therefore, $[b] \in W(A)$.
\end{proof}

\section{Central factorization}
\label{sec:Central}
A powerful way to use finite nuclear dimension (in the separable case) is via an exact factorization of the canonical embedding $A \hookrightarrow A_\omega$ using order zero maps into ultraproducts of finite dimensional $\mathrm C^*$-algebras (as proven in \cite[Proposition 2.2]{Robert:dimNucComp}, using \cite[Proposition 3.2]{WinterZacharias:NucDim}).
Here, we show that a similar factorization for $\F(B,A)$ may be made when $B \subset A_\omega$ is a separable $\mathrm C^*$-subalgebra of finite nuclear dimension.
The finite dimensional $\mathrm C^*$-algebras in the ultraproducts, however, become replaced by direct sums of hereditary subalgebras of $A$.
This factorization result can (and will) be applied to push certain regularity properties of $A$ to $\F(B,A)$ (just as $0$-comparison for finite dimensional $\mathrm C^*$-algebras gets pushed to $m$-comparison for a $\mathrm C^*$-algebra, by the first-named author in \cite{Robert:dimNucComp}).
Before stating the factorization result, we introduce notation.

Let $A$ be a $\mathrm{C}^*$-algebra.
Let $c\in M(A)_+$ be a contraction.
Define the c.p.c.\ map
$q_c\colon A\to \overline{cAc}$ given by $q_c(x)=c^{\frac 1 2}xc^{\frac 1 2}$.
If $\Sigma\subset M(A)_+$ is a finite set of positive contractions then we define $\mathbf C_\Sigma:=\bigoplus_{c\in \Sigma} \overline{cAc}$ and  $\mathbf Q_\Sigma\colon A\to \mathbf C_\Sigma$ by
\begin{align}\label{QSigma}
\mathbf Q_\Sigma=\bigoplus_{c\in \Sigma} q_c.
\end{align}
We may write $\mathbf C_\Sigma^A$ and $\mathbf Q_\Sigma^A$ if there is ambiguity in the choice of the  ambient $\mathrm{C}^*$-algebra. 

For a sequence of finite sets
$\Sigma_n\subset M(A)_+$ of positive contractions, define
\[
\mathbf C^A_{ (\Sigma_n)_n}:=\prod_\omega \mathbf C^A_{\Sigma_n}
\]
and set 
\[ \mathbf Q^A_{(\Sigma_n)_n}:=\pi_\omega \circ (\mathbf Q^A_{\Sigma_1}, \mathbf Q^A_{\Sigma_2},\cdots)\colon A\to \mathbf C^A_{ (\Sigma_n)_n} \]

Now, suppose that $B \subseteq A$ and we have a sequence of finite sets $\Sigma_n \subset B$.
Then the restriction of $\mathbf Q^A_{(\Sigma_n)_n}$
to $A\cap B'$ is of order zero, and factors through $(A\cap B')/B^\perp$.
Let us denote by $\widetilde {\mathbf Q}_{(\Sigma_n)_n}\colon (A\cap B')/B^\perp\to \mathbf C_{(\Sigma_n)_n}$ the factor map.

Here is the main result to be proven in this section.

\begin{theorem}
\label{thm:centralfactorization}
Let $A$ be a $\mathrm{C}^*$-algebra and let $B\subset A$ be a separable $\mathrm{C}^*$-subalgebra of nuclear dimension $m$. For each $k=0,1,\dots,2m+1$ there exist maps
\[
(A\cap B')/B^\perp\stackrel{ Q_{k}}{\longrightarrow} C_{k}\stackrel{R_{k}}{\longrightarrow} (A_\omega \cap B')/B^\perp,
\]
such that
\begin{enumerate}
\item
For each $k$, there exists a sequence $(\Sigma^k_n)_{n=1}^\infty$, where each $\Sigma_n^k\subset B$ is a finite set of positive contractions, such that 
$C_k=\mathbf{C}_{(\Sigma_n^k)_n}$ and $Q_k=\widetilde{\mathbf Q}_{(\Sigma_n^k)_n}$.
In particular, $Q_k$ is a c.p.c.\ map of order zero.

\item
For each $k$, $R_k$ is a c.p.c.\ map of order zero.

\item
For all $a\in (A\cap B')/B^\perp$ we have
\[
a=\sum_{k=0}^{2m+1} R_k Q_k(a).
\] 
\end{enumerate}
\end{theorem}

\begin{remark}
\label{rmk:centralfactorization-seqalg}
Suppose that $A$ is an ultraproduct algebra. 
Then, given $C_k, Q_k,R_k$ as in Theorem \ref{thm:centralfactorization}, we may improve our lot somewhat, using the diagonal sequence argument (cf.\ \cite[Section 4.1]{UnitlessZ}) as follows.
Given a separable subset $D$ of $(A \cap B')/B^\perp$ and for each $k$, a separable subset $C_k'$ of $C_k$ containing $Q_k(D)$, there exist *-linear maps
\[ \hat R_k\colon C_k \to (A \cap B')/B^\perp \]
such that:
\begin{enumerate}
\item $\hat R_k|_{C'_k}$ is c.p.c.\ order zero, and
\item $a = \sum_{k=0}^{2m+1} \hat R_k  Q_k(a)$ for all $a \in D$.
\end{enumerate}
Noting also that $C_k$ is a hereditary subalgebra of $A_\omega$, this proves Theorem \ref{thm:centralfactorizationsimple}.
\end{remark}

\vspace*{2mm}

The maps $R_k$ in the above theorem come chiefly from maps $\chi_\phi$ that we define presently.
Let $\phi\colon M_p(\C)\to A$ be a c.p.c.\ map of order zero map and set
$c=\phi(e_{11})$. Let us define a homomorphism $\chi_\phi\colon \mathrm{her}(c)\to A$ by
\[
\chi_\phi(x)=\sum_{i=1}^p \pi_{\phi}(e_{i1})x\pi_{\phi}(e_{1i}).
\]

\begin{lemma}
\label{lem:CpcCommute}
Let  $\phi\colon M_p(\C)\to A$ be  a c.p.c.\ map of order zero and let $c, \chi_\phi$ be as  defined above. 
For each contraction $a\in A$ we have  
\begin{equation}
\label{CpcCommuteEq}
\|[a,\phi^{1/2}]\|<\epsilon \Rightarrow \chi_\phi q_c(a)\approx_{3\epsilon} \phi(1)a.
\end{equation}
\end{lemma}

\begin{proof}
We have
\begin{align}
\notag
 \phi(1)a &= \int_{u \in U(M_p(\C))} \phi^{1/2}(u)^*\phi^{1/2}(u)a\, du \\
\notag
 &\approx_\epsilon \int_{u \in U(M_p(\C))} \phi^{1/2}(u)^*a\phi^{1/2}(u)\, du \\
\label{CpcCommuteApprox1}
&= \frac1p \sum_{i,j=1}^p \phi^{1/2}(e_{ij})a\phi^{1/2}(e_{ji}).
\end{align}

Now, note that, for $\eta > 0$,
\begin{align*}
\phi^{1/2}(e_{ij})a\phi^{1/2}(e_{ji}) &\approx_{\eta^{1/2}} g_{0,\eta}(\phi)(e_{i1})\phi^{1/2}(e_{1j})a\phi^{1/2}(e_{ji}) \\
&\approx_\epsilon g_{0,\eta}(\phi)(e_{i1})a\phi^{1/2}(e_{1j})\phi^{1/2}(e_{ji}) \\
&= g_{0,\eta}(\phi)(e_{i1})a\phi^{1/2}(e_{11})\phi^{1/2}(e_{1i}) \\
&\approx_\epsilon g_{0,\eta}(\phi)(e_{i1})\phi^{1/2}(e_{11})a\phi^{1/2}(e_{1j}) \\
&\approx_{\eta^{1/2}} \phi^{1/2}(e_{i1})a\phi^{1/2}(e_{1i}), 
\end{align*}
and since $\eta$ is arbitrary,
\[ 
\phi^{1/2}(e_{ij})a\phi^{1/2}(e_{ji}) \approx_{2\epsilon} \phi^{1/2}(e_{i1})a\phi^{1/2}(e_{1i}).
\]
It follows that, for each $i$,
\[ \frac1p \sum_{j} \phi^{1/2}(e_{ij})a\phi^{1/2}(e_{ji}) \approx_{2\epsilon} \phi^{1/2}(e_{i1})a\phi^{1/2}(e_{1i}). 
\]
Finally, by orthogonality of the errors, it follows that
\begin{equation}
\label{CpcCommuteApprox2}
 \frac1p \sum_{i,j} \phi^{1/2}(e_{ij})a\phi^{1/2}(e_{ji}) \approx_{2\epsilon} \sum_i \phi^{1/2}(e_{i1})a\phi^{1/2}(e_{1i}).
\end{equation}
Combining \eqref{CpcCommuteApprox1} and \eqref{CpcCommuteApprox2} produces \eqref{CpcCommuteEq}.
\end{proof}

\begin{remark*}
It is simpler to show
 that
\[
\|[a,\phi]\|<\epsilon \Rightarrow \chi_\phi q_c(a)\approx_{p\epsilon} \phi(1)a,
\]
and this estimate (differing in that the approximation on the left depends on the matrix size $p$) ultimately suffices for our application.
Nevertheless, the stronger estimate seems independently interesting.
\end{remark*}

\begin{lemma}
\label{lem:CpcFuncCalcCommute}
Given $f \in C_0((0,1])_+$ and $\e > 0$, there exists $\dl > 0$ such that the following holds:
If $\beta\colon D \to A$ is a c.p.c.\ order zero map between $\mathrm{C}^*$-algebras $D$ and $A$, where $D$ is unital, and $a \in A$ is a contraction which satisfies
\[ \|[a,\beta]\| < \dl \]
then 
\[ \|[a,f(\beta)]\| < \e. \]
\end{lemma}

\begin{proof}
Let $g \in C_0((0,1])_+$ be such that
$f \approx_{\e/4} g\cdot \id_{[0,1]}$. 
Then (by approximating $g$ by polynomials), we may find
$ 0<\dl<\frac\e{4\|g\|} $
such that, if $a,b$ are elements of a $\mathrm{C}^*$-algebra such that $b$ is a positive contraction and
$\|[a,b]\| < \dl$
then $\|[a,g(b)]\| < \e/4$.

Now, suppose that we have $\beta$ and $a$ as in the statement of the lemma.
We compute, for a contraction $x \in D$,
\begin{align*}
af(\beta)(x) &\approx_{\frac\e4} ag(\beta(1))\beta(x) \\
&\approx_{\frac\e4} g(\beta(1))a\beta(x) \\
&\approx_{\|g\|\frac\e{4\|g\|}} g(\beta(1))\beta(x)a \\
&\approx_{\frac\e4} f(\beta)(x)a.\qedhere
\end{align*}
\end{proof}

The proof of the following lemma contains the basic construction upon which the other results will be built.

\begin{lemma}\label{lem:basiccentral}
Let $D$ be a finite dimensional $\mathrm{C}^*$-algebra. 
For each $\epsilon>0$ there exists $\delta>0$ with the following property:

If $A$ is a $\mathrm{C}^*$-algebra and $\beta\colon D\to A$ is a c.p.c.\ map of order zero, then there exist maps 
\[
\xymatrix{
A\ar[r]^{Q_k} & C_k\, \ar@{^{(}->}[r]^{R_k} &A\\
}
\]
with $k=0,1$, with the following properties:
\begin{enumerate}
\item
For each $k=0,1$, there exists a finite set of  positive contractions $\Sigma_k\subset   C^*(\mathrm{im}(\beta))^\sim$ such that $C_k=\mathbf C_{\Sigma_k}^A$ and $Q_k=\mathbf Q_{\Sigma_k}^A$ (defined as in \eqref{QSigma}).
 
\item
For each $k=0,1$, $R_k$ is an injective *-homomorphism. Furthermore,  there exists  $h_k \in C_0((0,1])_+$ $\|h_k-\id_{[0,1]}\|<\epsilon$   and such that 
$[R_k,h_k(\beta)]=0$. 

\item
If $a\in A$ is a contraction such that 
$\|[a,\beta]\|<\delta$ then 
\[
R_0Q_0(a)+R_1Q_1(a)\approx_\epsilon a.
\]

\item $\|[R_k,\beta]\| < \epsilon$.
\end{enumerate}
\end{lemma}

\begin{proof}
Let us take a partition of unity $F = F_0 \amalg F_1$ for $C([0,1])$, consisting of positive elements whose supports each have diameter at most $\e$, and such that for each $k=0,1$, the elements of $F_k$ have pairwise disjoint (closed) supports.
It follows that there exists $h_k \in C_0((0,1])^+$ which is constant on the support of each element of $F_k$, and such that
\[ \|h_k - \id_{[0,1]}\| \leq \e. \]
Observe that, for $k=0,1$ and $f \in F_k$, since $h_k$ is constant on the support of $f$, for any positive contraction $y$ in a $\mathrm{C}^*$-algebra (which below we will take to be $\beta(1)$), we have
\begin{equation}
\label{basiccentral-hkCommuteEq}
 [h_k(y), \her(f(y))] = 0.
\end{equation}

We may assume, without loss of generality, that there exists $f_0 \in F_0$ such that $f_0(1)=1$; therefore, $F \setminus \{f_0\} \subset C_0((0,1])$.
Let $D=\bigoplus_{i=1}^q M_{n_i}(\C)$.
By Lemma \ref{lem:CpcFuncCalcCommute}, let $\dl > 0$
be such that, if $\beta$ is a c.p.c.\ order zero map from a unital $\mathrm{C}^*$-algebra to a $\mathrm{C}^*$-algebra containing a contraction $a$, and $\|[a,\beta]\| < \dl$ then $\|[a,f(\beta)^{1/2}]\| < \frac\e{6q|F|}$ for all $f \in F \setminus \{f_0\}$, and additionally, $\|[a,f_0(\beta(1))^{1/2}]\| < \frac\e2$.

For each $i=1,2,\dots,q$ and $f\in F \setminus \{f_0\}$, define
\begin{align*}
\beta_{f,i}&:=f(\beta)|_{M_{n_i}(\C)}\colon M_{n_i}(\C) \to A, \text{ and} \\
c_{f,i}&:=\beta_{f,i}(e_{11}) \in A^+.
\end{align*}
Set $c_0 := f_0(\beta(1)) \in A^\sim$.
Define
\begin{align*}
\Sigma_0 &:=\{c_{f,i}\mid f\in F_0 \setminus \{f_0\},i=1,\dots,q\}\cup \{c_0\}, \text{ and} \\
\Sigma_1 &:=\{c_{f,i}\mid f\in F_1,i=1,\dots,q\}.
\end{align*}
Let us define $Q_0,Q_1,C_0,C_1$ accordingly as in the statement of the lemma.
Let us define $R_k\colon Q_k \to A$ by
\begin{align*}
R_0((b_c)_{c \in \Sigma_0})& =
\sum_{\stackrel{f\in F_0}{f\neq f_0}}\sum_{i=1}^q \chi_{\beta_{f,i}}(b_{c_{f,i}})
+b_{c_0}, \\
R_1((b_c)_{c \in \Sigma_1})&=
\sum_{f\in F_1}\sum_{i=1}^q \chi_{\beta_{f,i}}(b_{c_{f,i}}).
\end{align*}
Notice that each
$R_k$ is a homomorphism since it is a sum of  homomorphisms with  orthogonal ranges.

(i) clearly holds by construction.
(ii) holds by \eqref{basiccentral-hkCommuteEq}.

To see (iii), let $a \in A$ be a contraction for which $\|[a,\beta]\|<\dl$.
By our choice of $\dl$ using Lemma \ref{lem:CpcFuncCalcCommute}, it follows that, for $f \in F \setminus \{f_0\}$,
\[ \|[a,f(\beta)^{1/2}]\| < \frac\e{6q|F|}, \]
whence by Lemma \ref{lem:CpcCommute},
\begin{equation}
\label{basiccentral-EndApprox1}
\beta_{f,i}(1)a \approx_{\frac\e{2q|F|}} \chi_{\beta_{f,i}} q_{c_{f,i}}(a).
\end{equation}
Also,
\begin{equation}
\label{basiccentral-EndApprox2}
 \|[a,f_0(\beta(1))^{1/2}]\| < \frac\e2. 
\end{equation}
We then compute
\begin{align*}
R_0Q_0(a) + R_1Q_1(a) &= \sum_{f \in F \setminus \{f_0\}} \sum_{i=1}^q \chi_{\beta_{f,i}} q_{c_{f,i}}(a) + c_0^{1/2}ac_0^{1/2} \\
&\approx^{\eqref{basiccentral-EndApprox1},\eqref{basiccentral-EndApprox2}}_\e \sum_{f \in F \setminus \{f_0\}} \sum_{i=1}^q \beta_{f_i}(1)a + f_0(\beta(1))(a) \\
&= \sum_{f \in F} f(\beta(1))a \\
&= a.
\end{align*}

(iv) follows from (ii), except with $2\epsilon$ in place of $\epsilon$.
Therefore, using $\epsilon/2$ instead of $\epsilon$ from the get-go will make (iv) hold as stated.
\end{proof}

\begin{proposition}
\label{prop:factorization}
Let $A$ be a $\mathrm{C}^*$-algebra and let $B\subseteq A$ be  a $\mathrm{C}^*$-subalgebra of   nuclear dimension at most $m$. Then for each finite set $F\subset B$ and $\epsilon>0$ there exist a finite set $G\subset B$, $\delta>0$, and maps
\[
\xymatrix{
A\ar[r]^{Q_k} & C_k\, \ar[r]^{R_k} &A,
}
\]
with $k=0,1,\dots,2m+1$  such that 
\begin{enumerate}
\item
For each $k$, there exists a   finite set of positive contractions $\Sigma_k\subset (B^\sim)_+\subseteq (A^\sim)_+$ such that
$C_{k} =\mathbf C_{\Sigma_k}^A$ and 
$Q_{k} =\mathbf Q_{\Sigma_k}^A$ (as defined in \eqref{QSigma}).

\item
For each $k$, the map $R_k$ is an order zero map and 
for every $f\in F$, we have $\|[f,R_k]\|<\epsilon$.

\item
If $a\in A$ is a contraction such that $\|[a,G]\|<\delta$ then 
\[
\sum_{k=0}^{2m+1}R_{k}Q_{k}(a)\approx_{\epsilon}a.
\]
\end{enumerate}
\end{proposition}

\begin{proof}
Set $\eta := \e/(6m+5)$.

Let us find an approximation of the identity map on $B$ within $(F,\eta)$ by c.p.c.\ maps 
\[
B\stackrel{\alpha_k}{\longrightarrow}D_k\stackrel{\beta_k}{\longrightarrow} B,
\] 
with $\beta_k$ of order zero and $k=0,1,\dots,m$.
Let $e$ be a positive contraction which approximately acts as an identity on $F$.
Set $e_k := \beta_k\alpha_k(e) \in B$ for $k=0,1\dots,m$ (the ``partition of unity"  of this decomposition).
By \cite[Lemma 3.4]{UnitlessZ} (cf.\ \cite[Proposition 4.2]{Winter:pure}), with an appropriate choice of $e$ and of the decomposition, we have 
\begin{equation}
\label{factorizationPOU}
\beta_k\alpha_k(a)\approx_{\eta} e_k a \ \forall k=0,\dots,m \quad \text{and} \quad a \approx_\eta \sum_{k=0}^m e_k a
\end{equation}
for all $a\in F$.

Let us apply Lemma \ref{lem:basiccentral} to each order zero map 
$\beta_k$ and with $\eta$ in place of $\e$. We obtain maps 
\[
A\stackrel{Q_{k,j}}{\longrightarrow}C_{k,j}\stackrel{R_{k,j}}{\longrightarrow} A,
\]
elements $h_{k,j} \in C_0((0,1])_+$ for $j=0,1$,
and a number $\delta_k>0$ satisfying (i)-(iv) of Lemma \ref{lem:basiccentral} for $\beta_k$ and $\eta$. 
Let us define 
\begin{equation}
\label{factorizationtRdef}
\widetilde R_{k,j} := h_{k,j}(e_k)R_{k,j}.
\end{equation}
 Notice that, by Lemma \ref{lem:basiccentral} (ii), $h_{k,j}(e_k)$ commutes with $R_{k,j}$, and therefore $\widetilde R_{k,j}$ is an order zero map.
Let us show that the data $Q_{k,j}$, $C_{k,j}$, and $\widetilde R_{k,j}$, with $k=0,1,\dots,m$ and $j=0,1$,  have the properties (i)-(iii) postulated by the proposition, for a suitable finite set $G\subset A$ and number $0<\delta<\min_k\delta_k$  to be determined soon.
(That is, the proposition as stated will follow by relabelling $(\widetilde R_{k,j})_{k=0,\dots,m,j=0,1}$ to $(R_k)_{k=0,\dots,2m+1}$.)

By Lemma \ref{lem:basiccentral} (i), (for $Q_{k,j}$, $C_{k,j}$, and $R_{k,j}$) 
property (i) is easily verified.

Let us show (iii). Let $a\in A$.
Since the image of each $\beta_k$ is finite-dimensional, we may find a finite subset $G$ of $A$ and a tolerance $\dl > 0$ such that $\|[a,G]\|<\delta$ implies that $\|[a,\beta_k]\|$ is sufficiently small, so that in turn by Lemma \ref{lem:basiccentral} (iii),
\[
a\approx_{\eta} R_{k,0}Q_{k,0}(a)+R_{k,1}Q_{k,1}(a).
\] 
for all $k$. Thus, multiplying by $\tilde e_{k,j}$ on both sides and summing over $k$ and $j$ we get
\begin{align*}
a &\approx^{\eqref{factorizationPOU}}_{\eta} \sum_{k=0}^m e_k (a) \\
&\approx_{2(m+1)\eta} \sum_{k=0}^m \sum_{j=0,1} e_kR_{k,j}Q_{k,j}(a) \\
&\approx^{\text{Lemma \ref{lem:basiccentral} (ii)}}_{2(m+1)\eta} \sum_{k=0}^m \sum_{j=0,1} \tilde{R}_{k,j}Q_{k,j}(a).
\end{align*}
as desired (since $(4m+5)\eta \leq \e$).

Finally, let us prove (ii). Let $f\in F$ and $b\in C_{k,j}$ be contractions.
Then
\begin{align*}
f\cdot \widetilde R_{k,j}(b) &\approx^{\eqref{factorizationtRdef}}_\eta f\cdot e_k\cdot R_{k,j}(b)\\
&\approx^{\eqref{factorizationPOU}}_{\eta} \beta_k\alpha_k(f)\cdot R_{k,j}(b)\\
&\approx^{\text{Lemma \ref{lem:basiccentral} (iv)}}_{\eta} R_{k,j}(b)\cdot \beta_k\alpha_k(f) \\
&\approx_{3\eta} \widetilde R_{k,j}(b)\cdot f,
\end{align*}
as desired (since $6\eta \leq \e$)
\end{proof}

The main theorem of this section will now be proven, essentially by turning approximate relations in the previous proposition, holding at the level of the algebra, into exact relations in the ultrapower algebra.

\begin{proof}[Proof of Theorem \ref{thm:centralfactorization}]
Let $(F_n)_{n=1}^\infty$ be an increasing sequence of finite sets with dense union in $B$. For each $F_n$  and with $\epsilon_n=1/n$, let us apply Proposition \ref{prop:factorization} to find  $\delta_n>0$, a finite $G_n\subset B$,  finite  sets $\Sigma_n^k\subset B^\sim$ with $k=0,1,\dots,2m+1$, and c.p.c.\ maps of order zero $R_n^k\colon C_k\to A$
with $k=0,1,\dots,2m+1$ that have the properties stated in the proposition.
In particular, we have that
\[
a\approx_{\frac 1 n}\sum_{k=0}^{2m+1} R_n^kQ_n^k(a)
\]
for all $a\in A$ such that $\|[a,G_n]\|<\delta_n$.
Drawing from an approximate identity, let $e_n \in B_+$ be such that
$\|[e_nae_n,G_n]\|<\delta_n$ for all contractions $a\in A\cap B'$ and
$c^{\frac12}e_n \approx_{\epsilon_n} (c^{\frac 1 2}e_n^2c^{\frac 1 2})^{\frac12}$
for all $c\in \Sigma_n^k$ and for all $k$.  Let $\widetilde\Sigma_n^k$ be the subset of $B$ given by 
$\widetilde\Sigma_n^k:=\{c^{\frac 1 2}e_n^2c^{\frac 1 2}\mid c\in \Sigma_n^k\}$.
Set 
\begin{align*}
\widetilde C_n^k &:=\mathbf C_{\tilde \Sigma_n^k}\subseteq C_n^k \text{ and} \\
 Q_n^k &:=\widetilde{\mathbf Q}_{(\Sigma_n^k)}.
\end{align*}
Then, for all contractions $a \in A \cap B'$,
\[
e_nae_n \approx_{\frac 1 n}\sum_{k=0}^{2m+1} R_n^kQ_n^k(e_nae_n)\approx_{4(m+1)\epsilon_n}
\sum_{k=0}^{2m+1} R_n^k Q_n^k(a).
\]
Define the map $\tilde R_k\colon \widetilde C_k \to A_\omega$ to be the one induced by
\[ (\widetilde R_1^k, \widetilde R_2^k, \dots)\colon \prod_n C_{\Sigma_n^k} \to \prod_n A. \]
By Proposition \ref{prop:factorization} (ii), the range of $R_k$ belongs to
$A_\infty\cap B'$. Furthermore, with $e=(e_n^2)_n\in A\cap B'$, we have
\[
a=ea=\sum_{k=0}^{2m+1} \widetilde R_k Q_k(a),
\]
so that $a=\sum_{k=0}^{2m+1} \widetilde R_k Q_k(a)$ modulo $B^\perp$, for all $a\in A\cap B'$.
\end{proof}

\section{Comparison in $\F(B,A)$}
\label{sec:Comp}

Here, we apply Theorem \ref{thm:centralfactorization} to gain an understanding of Cuntz comparison in a central sequence algebra $A_\omega \cap B'$: specifically, when $B$ has finite nuclear dimension, we are able to deduce Cuntz comparison in $A_\omega \cap B'$ from appropriate Cuntz comparisons in $A_\omega$ (at a cost which scales with the nuclear dimension of $B$).
This allows us to prove that $\F(B,A)$ has $M$-comparison for some $M$, provided that $B$ has finite nuclear dimension and $A$ has $m$-comparison for some $m$.
It also allows us to better understand fullness in $\F(A)$, when $A$ is simple, has finite nuclear dimension, and has at most one trace.

The first two results will set up notation, allowing us to state the main result, Proposition \ref{prop:CommComparison}.
The proof of Proposition \ref{prop:CommComparison} uses the full strength of Theorem \ref{thm:centralfactorization}, in the sense that the specific form of the maps $Q_k$ in the factorization is used. 

\begin{lemma}
\label{lem:CutdownMult}
Let $a,c$ be two commuting positive contractions, and let $\lambda > 0$.
Then
\[ [(a-\lambda)_+(c-\lambda)_+] \leq [(ac-\lambda)_+] \leq [(a-\lambda^{1/2})_+(c-\lambda^{1/2})_+]. \]
\end{lemma}

\begin{proof} The $\mathrm{C}^*$-algebra
$C^*(a,c)$ is commutative, and hence isomorphic to $C_0(X)$ for some $X$.
Since for $f,g\in C_0(X)_+$, we have $[f] \leq [g]$ iff $\forall x \in X, f(x) > 0 \Rightarrow g(x) > 0$, it suffices to prove the lemma assuming that $a$ and $c$ are scalars.
This is not difficult.
\end{proof}

\begin{lemma}
\label{lem:precBChar}
Let $A$ be a $\mathrm{C}^*$-algebra and let $B$ be a $\mathrm{C}^*$-subalgebra of $A$.
Let $a,b\in A\cap B'$ be positive elements.
Consider the following relations between $a$ and $b$.
\begin{enumerate}[(i)]
\item For each $\e>0$, there exists $\dl > 0$ such that
\[ [(ac-\e)_+] \leq [(bc-\dl)_+] \]
in $\Cu(A)$, for all positive contractions $c \in B_+$.
\item For each $\e>0$, there exists $\dl > 0$ such that
\[ [(a-\e)_+(c-\e)_+] \leq [(b-\dl)_+(c-\dl)_+] \]
in $\Cu(A)$, for all positive contractions $c \in B_+$.
\item $[a] \leq [b]$ in $\Cu(A \cap B')$.
\end{enumerate}
Then (i) $\Leftrightarrow$ (ii) $\Leftarrow$ (iii).
\end{lemma}

We shall write $a\preceq_B b$ if the equivalent conditions (i) and (ii) hold.
In this case we say that $a$ is Cuntz smaller than $b$ by cutdowns of elements from $B$.

\begin{remark*}
If $a=1$ then (i) holds so long as it holds for one single $\e > 0$.
Certainly, suppose that $\e_0,\dl_0 > 0$ are such that, for any $c \in B_+$, $[(c-\e_0)_+] \leq [(bc-\dl_0)_+]$.
Given any other $\e > 0$, set $\eta := \frac{2\e}{\e_0+1}$, so that $[(c-\e)_+] = [(g_\eta(c)-\e_0)_+]$ and $[(bg_\eta(c)-\dl_0)] \leq [(bc-\frac{\eta\dl_0(\dl_0+1)}2)_+]$ (proven in the same way as Lemma \ref{lem:CutdownMult}), so that the condition in (i) holds with $\dl:= \frac{\eta\dl_0(\dl_0+1)}2$.
\end{remark*}

\begin{proof}
The equivalence (i) $\Leftrightarrow$ (ii) is immediate from Lemma \ref{lem:CutdownMult}.

(iii) $\Rightarrow$ (ii):
Suppose that $[a] \leq [b]$ in $\Cu(A \cap B')$.
Then given $\e > 0$, there exists $\dl > 0$ and $x \in A \cap B'$ such that $(a-\e)_+ = x(b-\dl)_+x^*$;
we may assume that $\dl < \e$.
Thus, for $c \in B_+$,
\[ (a-\e)_+(c-\e)_+ = x(b-\dl)_+x^*(c-\e)_+ = x(b-\dl)_+(c-\e)_+x^*, \]
whence $[(a-\e)_+(c-\e)_+] \leq [(b-\dl)_+(c-\e)_+] \leq [(b-\dl)_+(c-\dl)_+]$.
\end{proof}

Here is the main result of this section, which shows that if $A$ is an ultraproduct algebra and $B$ has finite nuclear dimension, then condition (i) of Lemma \ref{lem:precBChar} implies a weakened version of (iii).

\begin{proposition}
\label{prop:CommComparison}
Let $B\subseteq A$ be $\mathrm{C}^*$-algebras, with $B$ separable of nuclear dimension $m$, and $A$ an ultraproduct algebra.
Let $a,b_k\in A\cap B'$, with $k=0,1,\dots,2m+1$ be positive elements such that 
$a\preceq_B b_k$ 
for all $k$.  Then
\[ [a] \leq \sum_{k=0}^{2m+1} [b_k] \]
in $\Cu(A\cap B')$.  

In particular, for $a,b \in A \cap B'$, $[a] \leq N[b]$ in $\Cu(A \cap B')$ for some $N \in \mathbb{N}$ if and only if $a \preceq_B 1_M \otimes b$ for some $M \in \mathbb{N}$.
\end{proposition}

\begin{proof}
By possibly adjoining a unit to $A$ and adding the unit of $A$ to $B$, we may assume that $B$ is a unital $\mathrm{C}^*$-subalgebra of $A$.
Let $ Q_k,\mathbf C_k,\widetilde R_k$ be as given by Theorem \ref{thm:centralfactorization}.

Given $\epsilon>0$, by hypothesis, there exists $\delta>0$ such that $[(ac-\epsilon)_+] \leq [(bc-\delta)_+]$ in $\Cu(A)$.
It follows that for each positive contraction $c \in B_+$, there exists $x_{k,c} \in A$ such that 
\[ (ac-2\e)_+ = x_{k,c}^*x_{k,c} \]
and
\[ g_{\dl}(b_kc)x_{k,c}=x_{k,c}. \]
In particular, $\|x_{k,c}\| \leq 1$ and $x_{k,c} \in \her(c)$.

Using the form of $ Q_k$, it follows that there exists $y_k \in C_k$ such that
\[ ( Q_k(a)-2\e)_+ = y_k^*y_k \]
and
\[ g_\dl( Q_k(b_k))y_k = y_k. \]
(Namely, we let $y_k = (y_{k,n})_{n=1}^\infty$ where $y_{k,n} = (x_{k,c})_{c \in \Sigma_n^k}$.)
Since $\e$ is arbitrary, we find that
\[ [ Q_k(a)] \leq [ Q_k(b_k)] \]
in $\Cu(C_k)$.
Therefore, we may find a separable subalgebra $C_k'$ of $C_k$ containing $ Q_k(a),  Q_k(b_k)$, and such that
\begin{equation}
\label{CommComparisonSepCu}
 [ Q_k(a)] \leq [ Q_k(b_k)]
\end{equation}
in $\Cu(C_k')$.

Using $D = \{a,b_0,\dots,b_{2m+1}\}$, obtain maps $\hat R_k\colon C_k \to (A_\infty \cap B')/B^\perp$ as in Remark \ref{rmk:centralfactorization-seqalg}.
By \eqref{CommComparisonSepCu}, and since $\hat R_k|_{C_k'}$ is order zero, $[\hat R_k  Q_k(a)] \leq [\hat R_k  Q_k(b_k)]$.

Thus, we have
\[ a = \sum_{k=0}^{2m+1} \hat R_k Q_k(a) \preceq \bigoplus_{k=0}^{2m+1} \hat R_k Q_k(b_k) \leq \bigoplus_{k=0}^{2m+1} b_k. \]
\end{proof}

We now derive some consequences of Proposition \ref{prop:CommComparison}.

\begin{proposition}
\label{prop:Fcomparison}
Suppose that $A$ has $M$-comparison and  $B\subseteq A_\infty$ is a separable $\mathrm{C}^*$-subalgebra of nuclear dimension at most $m$. Then $\F(B,A)$  has $(2(M+1)(m+1)-1)$-comparison.
\end{proposition}

\begin{proof}
Let us suppose that $(k+1)[a]\leq k[b_{i}]$ in the Cuntz semigroup of $\F(B,A)$, with $i=1,2,\dots,2(M+1)(m+1)$ and for some $k\in \N$. By Lemma \ref{lem:precBChar}
$(k+1)[a]\leq_B k[b_i]$ in $\Cu(A_\omega)$, for all $i$. Thus, given $\e>0$,   there exists $\delta>0$ such that  
for each positive contraction $c\in B_+$, we have $(k+1)[(ac-\epsilon)_+]\leq_B k[(b_i-\delta)_+]$ in $\Cu(A_\omega)$. 
By Proposition \ref{prop:DivCompSequence}, the $\mathrm{C}^*$-algebra $A_\omega$ has $M$-comparison, so that for each $1\leq i\leq 2(M+1)(m+1)-M$, we get  
$[(ac-\epsilon)_+]\leq [\sum_{j=i}^{i+M} (b_jc-\delta)_+]$. This, combined with Proposition \ref{prop:CommComparison}, implies that $[a]\leq \sum_{i=1}^{2(M+1)(m+1)} [b_i]$, as desired.
\end{proof}

In the remainder of this section, we explore some easy consequences of Proposition \ref{prop:CommComparison}
to fullness in $\F(A)$ for simple unital $\mathrm{C}^*$-algebras $A$, particularly those with unique trace.
These consequences will not be used in the sequel.
In ongoing work, the authors are further pursuing the problem of determining when an element of $\F(A)$ is full.

\begin{lemma}
\label{lem:AutomaticFullness}
Let $A$ be a unital $\mathrm{C}^*$-algebra with finite nuclear dimension.
The following are equivalent:
\begin{enumerate}[(i)]
\item For all $a \in \F(A)$, $a$ is full in $A_\omega$ if and only if it is full in $\F(A)$;
\item For all $a \in \F(A)_+$, if $a$ is full in $A_\omega$ then there exists $\gamma_a > 0$ such that
\[ \tau(ac) \geq \gamma_a\tau(c) \]
for all $c \in A_+$ and $\tau \in \mathrm{QT}(A_\omega)$.
\end{enumerate}
\end{lemma}

\begin{proof}
(i) $\Rightarrow$ (ii):
Suppose that $a \in \F(A)_+$ is full in $A_\omega$.
Then by (i) it is full in $\F(A)$, and so there exist $x_1,\dots,x_k \in \F(A)$ such that
$1 = \sum_{i=1}^k x_iax_i^*$.
Hence, for each $c \in A_+$ and $\tau \in \mathrm{QT}(A_\omega)$,
\begin{align*}
\tau(c) &= \sum_{i=1}^k \tau(x_iacx_i^*) \leq \sum_{i=1}^k \|x_i\|^2 \tau(ac),
\end{align*}
and therefore (ii) holds upon setting $\gamma_a = (\sum_{i=1}^k \|x_i\|^2)^{-1}$.

(ii) $\Rightarrow$ (i):
Suppose that $a \in \F(A)$ is full in $A_\omega$, and let us show that it is full in $\F(A)$.
Without loss of generality, let us assume that $a \geq 0$.
Let $\eta > 0$ be such that $g_\eta(a)$ is still full in $A_\omega$, and then let $K \in \N$ be such that $K > \gamma_{g_\eta(a)}^{-1}$. Let $m$ denote the nuclear dimension of $A$.
We shall show that $1 \preceq_A a \otimes 1_{(m+1)(K+1)}$, from which it follows by Proposition \ref{prop:CommComparison} that $a$ is full.

Certainly, for $c \in A_+$ and $\tau \in \mathrm{QT}(A_\omega)$, we have
\begin{align*}
d_\tau((c-\eta)_+) &\leq \tau(g_\eta(c)) \\
&\leq K\tau(g_\eta(a)g_\eta(c)) \\
&\leq K\tau(g_{\frac{\eta^2}4}(ac)) \\
&\leq Kd_\tau((ac-\frac{\eta^2}8)_+),
\end{align*}
which implies that $[(c-\eta)_+] <_s (K+1)[(ac-\frac{\eta^2}8)_+]$.
By Proposition \ref{prop:DivCompSequence}, it follows that $[(c-\eta)_+] \leq (m+1)(K+1)[(ac-\frac{\eta^2}8)_+]$
in $\Cu(A_\omega)$.
Thus, by the remark following Lemma \ref{lem:precBChar}, $1 \preceq_A a \otimes 1_{(m+1)(K+1)}$, as required.
\end{proof}

\begin{theorem}
Let $A$ be a simple unital separable $\mathrm{C}^*$-algebra with finite nuclear dimension and a unique tracial state.
Then for $a \in \F(A)$, $a$ is full in $A_\omega$ if and only if it is full in $\F(A)$.
\end{theorem}

\begin{remark*}
By \cite[Theorem 1.1]{MatuiSato:dr}, if $A$ is unital, simple, separable, nuclear, and quasidiagonal, and has a unique tracial state, then it automatically has finite nuclear dimension, so this theorem applies.
\end{remark*}

\begin{proof}
We shall use $\mu$ to denote both the unique tracial state on $A$ and its extension to $A_\omega$ (given by taking its limit).
It suffices to assume that $a\in \F(A)$
is positive. We shall verify that Condition (ii) of Lemma \ref{lem:AutomaticFullness} holds
with 
\[
\gamma_a=\inf\{\tau(a)\mid \tau\in \mathrm{QT}(A_\omega),\,  \tau(1)=1\}
\] 
(which is positive by the fullness of $a$ in $A_\omega$).

Let $\tau\in \mathrm{QT}(A_\omega)$.
Since $a$ is central, $\sigma(c) := \tau(ac)$, with $c\in A$,
defines a quasitrace $\sigma\colon A \to \mathbb{C}$.
Since $A$ is exact  and has a unique trace (up to a scalar multiple), we find that
$\sigma = \sigma(1)\cdot \mu(\cdot)$. Plugging $c\in A_+$ on both sides and using that $\sigma(1)=\tau(a)$ we get  
\begin{align*}
\tau(ac) =\tau(a)\mu(c).
\end{align*}
If $\tau(1)=\infty$ then $\tau(a)=\infty$ (since $a$ is full in $A_\omega$).  
So $\tau(ac) =\tau(a)\mu(c)$ clearly implies that  $\tau(ac)\geq \gamma_a\tau(c)$.
Otherwise, assume that $\tau(1)=1$. Then the restriction of $\tau$ to $A$ agrees with $\mu$, and so
$\tau(ac)=\tau(a)\mu(c)=\tau(a)\tau(c)\geq \gamma_a\tau(c)$, as required.  
\end{proof}

\section{Divisibility up to cancellation in $\F(B,A)$}
\label{sec:Divis}

In this section, we establish the following.

\begin{theorem}
\label{thm:TracialDiv}
Let $A$ be a $\mathrm{C}^*$-algebra and let $B \subset A_\omega$ be a separable $\mathrm{C}^*$-subalgebra.
Suppose that $A$ is $N$-divisible for some $N \in \N$, and that $\mathrm{dim}_\mathrm{nuc} B < \infty$.
Then for any $k \in \N$ and any $\e > 0$, there exists a c.p.c.\ order zero map 
$\phi\colon M_k(\C) \to \F(B,A)$ such that
\begin{align}
\label{TracialDivEq1}
d_\tau(1-\phi(1_k)) \leq \e d_\tau(1)
\end{align}
for every quasitrace $\tau \in \mathrm{QT}(\F(B,A))$ and 
\begin{align}
\label{TracialDivEq2}
d_\tau(\phi(e_{11})) &\geq  \Big(\frac 1k -\e\Big)d_\tau(1)
\end{align}
for every bounded  $\tau\in \mathrm{QT}(\F(B,A))$.
\end{theorem}

Note that, by the following lemma applied to $\F(B,A)$, \eqref{TracialDivEq2} can be reformulated as saying that
\[ L[1]+p[1] \leq L[1] + q[\phi(e_{11})], \]
for some $L,p,q \in \N$, where $\frac pq$ can be taken to be arbitrarily close to $\frac 1k$.
Thus, in the presence of appropriate cancellation properties, it would follow that $\phi(e_{11})$ is (controllably) full, entailing divisibility of the unit of $\F(B,A)$.

\begin{lemma}
\label{lem:TrAlgebraic}
Let $A$ be a unital $\mathrm{C}^*$-algebra, let $[a] \in \Cu(A)$, and $\gamma > 0$.
Then $d_\tau(1)\leq \gamma' d_\tau(a)$ for all bounded quasitraces $\tau\in \mathrm{QT}(A)$ and  some $0<\gamma'<\gamma$   if and only if $L[1] + p[1] \leq L[1] + q[a]$ for some $L,p,q \in \N$ with $\frac qp < \gamma$.
\end{lemma}

\begin{proof}
The reverse direction is an easy computation.
For the forward direction, suppose that $d_\tau(1)\leq \gamma' d_\tau(a)$ for all bounded $\tau\in \mathrm{QT}(A)$.
Let us choose   $p_0,q_0 \in \N$ such that 
$\gamma'<\frac {q_0}{p_0} < \gamma$. Then $d_\tau(1)< \frac {q_0}{p_0} d_\tau(a)$ for all  bounded $\tau\in \mathrm{QT}(A)$ such that $0<d_\tau(a)<\infty$.
It follows that $d_\tau((p_0+q_0)[1]) < d_\tau(q_0([1]+[a]))$  for each $\tau\in \mathrm{QT}(A)$ such that $0<d_\tau([1]+[a])<\infty$; note that also $(p_0+q_0)[1] \propto q_0([1]+[a])$.
Thus, by \cite[Proposition 2.1]{OrtegaPereraRordam} (essentially \cite[Lemma 4.1]{GoodearlHandelman:Ranks}), it follows that, for some $k \in \N$, $k(p_0+q_0)[1] \leq kq_0([1]+[a])$.
Now, set $L:= kq_0$, $p:= kp_0$ and $q := kq_0$.
\end{proof}

The proof of Theorem \ref{thm:TracialDiv} is broken into two steps.
First, in Lemma \ref{lem:TrDiv}, we establish a much weaker form of the conclusion of Theorem \ref{thm:TracialDiv}, where $\phi(1_k)$ is full up to cancellation, but the degree of fullness does depend on $N$ and $\mathrm{dim}_\mathrm{nuc} B$.
Then, we use a technique to minimize the defect, removing this dependence.

\begin{proposition}
\label{prop:VeryWeakDiv}
Let $A$ be $N$-almost divisible and let $B \subseteq A_\omega$ be separable and of nuclear dimension at most $m$.
Let $d_0,\dots,d_{2m+1} \in \F(B,A)_+$, $k \in \N$ and $\e > 0$.
Suppose that there exist $[a_1],\dots,[a_R], [b_1],\dots,[b_R] \in \Cu(A)$ and $K_1,\dots,K_R \in \N$ such that
\[ [a_j] \leq K_j[d_i] + [b_j] \]
for all $i=0,\dots,2m+1$ and $j=1,\dots,R$.

Then there exist c.p.c.\ order zero maps $\phi_0,\dots,\phi_{2m+1}\colon M_k(\C) \to \F(B,A)$ such that
\begin{enumerate}
\item $\phi_i(M_k(\C)) \subseteq \her(d_i)$ for each $i$; and
\item
for each $j$,
\[ [(a_j-\e)_+] \leq K_j (N+1)(k+1) \sum_{i=0}^{2m+1} [\phi_i(e_{11})] + (2m+2)[b_j]. \]
\end{enumerate}
\end{proposition}

\begin{proof}
Let us apply  Theorem \ref{thm:centralfactorization}, with $A_\omega$ in place of $A$, to obtain $\mathrm{C}^*$-algebras $C_i$, and maps $Q_i$ and $R_i$, with $i=0,\dots,2m+1$, as in the statement of that theorem. For each $i$, $Q_i$ is an order zero map. Thus, $[Q_i(a_j)] \leq K_j[Q_i(d_i)] + [Q_i(b_j)]$ for each  $j=1,\dots,R$.
Let $\dl > 0$ be such that, for all $i,j$,  
\[ 
[(Q_i(a_j) - \frac\e{2m+1})_+] \leq K_j[(Q_i(d_i)-\dl)_+] + [Q_i(b_j)]. 
\]
Since the $\mathrm{C}^*$-algebra $C_i$ is $N$-almost divisible, there exists  $\psi_i\colon M_k(\C) \to \her(Q_i(d_i))$ of order zero and such that $[(Q_i(d_i)-\dl)_+] \leq (N+1)(k+1)[\psi_i(e_{11})]$.
It follows that 
\begin{align}\label{LargepsisEq}
[(Q_i(a_j)-\frac\e{2m+1})_+] \leq K_j(N+1)(k+1)[\psi_i(e_{11})] + [Q_i(b_j)]
\end{align}
for each $j$.

Let $C_i'$ be a unital separable $\mathrm{C}^*$-subalgebra of $C_i$ which contains $Q_i(d_i)$, $Q_i(a_j)$ and  $Q_i(b_j)$ for all $j$, and all of $\psi_i(M_k(\C))$. Notice that $\psi_i(M_k(\C)) \in \overline{Q_i(d_i) C_i' Q_i(d_i)}$.  Furthermore, we may enlarge $C_i'$ if necessary -- while retaining its separability -- so that \eqref{LargepsisEq} holds in $\Cu(C_i')$.
Let us use Remark \ref{rmk:centralfactorization-seqalg}, with $C_i'$ as just described and $D := C^*(\{1\} \cup \{d_i\} \cup \{a_j,b_j\}) \subseteq \F(B,A)$, to obtain  $\hat R_i\colon C_i \to \F(B,A)$  such that $\hat R_i|_{C'_i}$ is a c.p.c.\ map of order zero and $a = \sum_{i=0}^{2m+1} \hat R_i Q_i(a)$ for all $a \in D$.

For each $i$, let us set $\phi_i:= \hat R_i \circ \psi_i\colon M_k(\C) \to \F(B,A)$. Let us show that these are c.p.c. order zero maps with the desired properties. 
Using the positivity of $\phi_i$, we find that 
\[
\phi_i(M_k(\C)) = \hat R_i(\psi_i(M_k(\C)))\subseteq \her(\hat R_i(Q_i(d_i))) \subseteq \her(d_i).\]
We note also that, for each $j$, since $a_j = \sum_{i=0}^{2m+1} \hat R_i Q_i(a_j) \approx_\e \sum \hat R_i((Q_i(a_j)-\frac\e{2m+1})_+)$,
\begin{align*}
[(a_j-\e)_+] &\leq \sum_{i=0}^{2m+1} [\hat R_i( (Q_i(a_j)-\frac\e{2m+1})_+)] \\
&\leq \sum_{i=0}^{2m+1} ((N+1)(k+1)[\hat R_i(\psi_i(e_{11}))] + [\hat R_i(Q_i(b_j))] \\
&= K_j (N+1)(k+1) \sum_{i=0}^{2m+1} [\phi_i(e_{11})] + (2m+2)[b_j].\qedhere
\end{align*}
\end{proof}

\begin{lemma}
\label{lem:TwoOrthogTrFull}
Let $A$ be $N$-almost divisible and let $B \subseteq A_\omega$ be separable and of nuclear dimension at most $m$.
Then there exist orthogonal positive elements $d_0,d_1 \in \F(B,A)_+$ such that
\begin{equation}
\label{TwoOrthogTrFullIneq}
d_\tau(1) \leq 4(m+1)(m+2)(N+1)d_\tau(d_i),
\end{equation}
for all bounded quasitraces $\tau$ on $\F(B,A)$ and for $i=0,1$.
\end{lemma}

\begin{proof}
Using $d_0=\cdots =d_{2m+1}:=1$ which satisfy $[1] \leq [d_i]$, and $k=2m+3$ in Proposition \ref{prop:VeryWeakDiv}, we obtain $\phi_0,\dots,\phi_{2m+1}\colon M_k(\C) \to \F(B,A)$ for which 
\[
[1] \leq (N+1)(2m+4)\sum_{i=0}^{2m+1} [\phi_i(e_{11})].
\]
Set $[a] := \sum_{i=0}^{2m+1} [\phi_i(e_{11})]$, so that
\[ 
[1] \leq (2m+2)(N+1)(2m+4)[a] \quad\text{and}\quad (2m+3)[a] \leq (2m+2)[1]. 
\]
Let $\e > 0$ be such that $[1] \leq 4(m+1)(m+2)(N+1)[(a-\e)_+]$.
Now let us define
\begin{align*}
d_0 &:=g_{\e}(a),\\
d_1 &:=1-g_{\frac \e 2}(a).
\end{align*}
Then $d_0$ and $d_1$ are orthogonal and $[1] \leq 4(m+1)(m+2)(N+1)[d_0]$ in the Cuntz semigroup of $\F(B,A)$.
Note that
\[ [1] \leq [d_1] + [g_{\frac\e2}(a)] \leq [d_1] + [a]. \]
Hence, by multiplying by $(2m+3)$, we get
\begin{align*}
(2m+3)[1] &\leq (2m+3)[d_1] + (2m+3)[a] \\
&\leq (2m+3)[d_1] + (2m+2)[1].
\end{align*}
Applying $d_\tau$, where $\tau$ is a bounded quasitrace, and then cancelling yields \eqref{TwoOrthogTrFullIneq}.
\end{proof}

\begin{lemma}
\label{lem:LftFacts}
Let $A$ be a unital $\mathrm{C}^*$-algebra and let $b,c \in A_+$ be positive commuting elements. Let $\gamma>0$.
If $d_\tau(1) \leq \gamma d_\tau(c)$ for every (bounded)   $\tau \in \mathrm{QT}(A \cap \{b\}')$ then $d_\tau(b) \leq \gamma d_\tau(bc)$ for every (bounded)   $\tau \in \mathrm{QT}(A)$.
\end{lemma}

\begin{proof}
For each (bounded) quasitrace $\tau$ on $A$, define $\hat\tau\colon (A \cap \{b\}')_+ \to [0,\infty]$ by $\hat\tau(x) := \sup \tau(b^{1/n}x)$.
It is easy to see that $\hat\tau$  is a (bounded) quasitrace, and so
\begin{align*}
d_\tau(b) = d_{\hat\tau}(1) \leq \gamma d_{\hat\tau}(c) =  \gamma d_\tau(bc),
\end{align*}
as required.
\end{proof}

\begin{proposition}
\label{prop:ManyOrthogTrFull}
Given $N,m \in \N$, there exists $P(N,m,i) \in \N$ for $i=0,1,\dots$ such that the following holds:
If  $A$ is $N$-almost divisible and   $B \subseteq A_\omega$ is separable and has nuclear dimension at most $m$, then there exist pairwise orthogonal positive elements $d_0,d_1,\dots \in \F(B,A)$ such that
\[ d_\tau(1) \leq P(N,m,i)d_\tau(d_i)\]
 for each $i$ and each bounded quasitrace $\tau$ on $\F(B,A)$.
\end{proposition}

\begin{proof}
This follows easily  using Lemmas \ref{lem:TwoOrthogTrFull} and \ref{lem:LftFacts}: We begin by getting two positive orthogonal elements $d^0_0,d^0_1 \in \F(B,A)$ satisfying \eqref{TwoOrthogTrFullIneq}.
We note that $\F(B,A) \cap \{d^0_0,d^0_1\} \cong \F(B',A)$ where $B':=\mathrm C^*(B \cup \{d^0_0,d^0_1\})$, and by \cite[Lemma 7.1]{UnitlessZ}, $\mathrm{dim}_\mathrm{nuc} B' \leq 2m-1$.
Thus by Lemma \ref{lem:TwoOrthogTrFull}, we get two more positive orthogonal elements $d^1_0,d^1_1 \in \F(B,A) \cap \{d^0_0,d^0_1\}'$ satisfying \eqref{TwoOrthogTrFullIneq} but with $2m-1$ in place of $m$.
Hence, $d_0:= d^0_0$, $d_1:=d^0_1d^1_0$, and $d^0_1d^1_1$ are positive orthogonal elements in $\F(B,A)$. Using Lemma \ref{lem:LftFacts}, we get
\[ d_\tau(1) \leq 4(m+1)(m+2)(N+1)\cdot4(m+2)(m+3)(N+1) d_\tau(d^0_1d^1_i) \]
for $i=0,1$.
The entire sequence $(d_i)_{i=1}^\infty$ is obtained by continuing in this manner, and we find that
\[ P(N,m,i) := 4(m+1)(m+2)(N+1)\Big(4(m+2)(m+3)(N+1)\Big)^i \]
works.
\end{proof}

\begin{lemma}
\label{lem:TrDiv}
Given $N, m \in \N$, there exists $Q(N,m)$ such that the following holds: If  $A$ is $N$-almost divisible and   $B \subset A_\omega$ is separable and has nuclear dimension at most $m$, then for each $k\in \N$  there exists a c.p.c.\ order zero map $\phi\colon M_k(\C) \to \F(B,A)$ such that
\begin{equation}
\label{TrDivIneq2}
d_\tau(1) \leq Q(N,m)d_\tau(\phi(1_k))
\end{equation}
for all bounded quasitraces $\tau$ on $\F(B,A)$.
\end{lemma}

\begin{proof}
Using the constants from Proposition \ref{prop:ManyOrthogTrFull}, set 
\[
P:=\max \{P(N,m,i)\mid i=0,\dots,2m+1\}\] 
and $Q(N,m):= 8P\cdot (N+1)$. 

Given $A$ and $B$ as in the statement, let us use Proposition \ref{prop:ManyOrthogTrFull} to get orthogonal positive elements $d_0,\dots,d_{2m+1} \in \F(B,A)_+$ such that $d_\tau(1) \leq  Pd_\tau(d_i)$ for each $i$ and each bounded quasitrace $\tau$ on $\F(B,A)$.
By Lemma \ref{lem:TrAlgebraic}, for each $i$ there exist $L_i,p_i,q_i \in \N$ such that $\frac {p_i}{q_i} > \frac 1 {2P}$ and 
\[
L_i[1]+p_i[1] \leq L_i[1]+q_i[d_i].\]
Setting $L:= \max_i L_i$, it follows that $L[1]+p_i[1] \leq L[1]+q_i[d_i]$ for all $i$.
Furthermore,  with $p:=\prod_{i=0}^{2m+1} p_i$ and $q := p \cdot \max_i \frac{q_i}{p_i}$, we have $\frac pq > \frac{1}{2P}$ and $L[1]+p[1] \leq L[1]+q[d_i]$ for all $i$.
From this, we obtain that 
\begin{equation}\label{TrDivStartIneq}
L[1]+np[1] \leq L[1]+nq[d_i]
\end{equation} 
for all $i$ and all $n=1,2,\dots$. Let us fix $n$ large enough (how large value will be specified soon).
Feeding $d_0,\dots,d_{2m+1}$ and \eqref{TrDivStartIneq} to Proposition \ref{prop:VeryWeakDiv}, we obtain c.p.c.\ order zero maps $\phi_0,\dots,\phi_{2m+1}\colon M_k(\C) \to \F(B,A)$ such that $\phi_i(M_k(\C)) \subseteq \her(d_i)$ and
\begin{equation}
\label{TrDivEndIneq}
 L[1]+np[1] \leq (2m+2)L[1] + 2nq(N+1) \sum_{i=0}^{2m+1} [\phi_i(1_k)].
\end{equation}
Since the $d_i$'s are orthogonal, it follows that $\phi := \sum_{i=0}^{2m+1} \phi_i$ is a c.p.c.\ order zero map.
Moreover, \eqref{TrDivEndIneq} implies that for each bounded quasitrace $\tau\in \mathrm{QT}(\F(B,A))$ we have
\[
\frac{L+np}{nq}d_\tau(1)\leq \frac{(2m+2)L}{nq}d_\tau(1)+2(N+1)d_\tau(\phi(1)).
\]
Observe that $\frac{1}{2P}<\frac{L+np}{nq}$ for all $n$ while $\frac{(2m+2)L}{nq}\to 0$ as $n\to \infty$.  
It is now clear that choosing $n$ large enough we will have
$\frac{1}{4P}d_\tau(1)\leq 2(N+1)d_\tau(\phi(1))$
for every bounded quasitrace $\tau$, as required.
\end{proof}

In the remainder of this section, we show how to get from the conclusion of Lemma \ref{lem:TrDiv} to the conclusion of Theorem \ref{thm:TracialDiv}.
This step can be stated in a very general form, as follows.

\begin{proposition}
\label{prop:MinimizingDefect}
Let $A$ be a $\mathrm{C}^*$-algebra, let $B \subset A_\omega$ be a separable $\mathrm{C}^*$-subalgebra, and let $k \in \N$.
Suppose that there exists $Q > 0$ such that, for every c.p.c.\ order zero map 
$\phi\colon M_k(\C) \to A_\omega \cap B'$ and $C:=C^*(B \cup \phi(M_k(\C)))$, there exists a c.p.c.\ order zero map $\psi\colon M_k(\C) \to \F(C,A)$ such that
\begin{equation}
\label{MinimizingDefect-Input}
d_\tau(1) \leq Qd_\tau(\psi(1)).
\end{equation}
for every bounded quasitrace $\tau$ on $\F(C,A)$,
Then, for every $\e > 0$, there exists a c.p.c.\ order zero map $\psi\colon M_k(\C) \to \F(B,A)$ such that
\begin{equation}
\label{MinimizingDefect-Output}
[1-\psi(1)] <_s \e [1].
\end{equation}
\end{proposition}

Prior to proving this proposition, let us see how to prove Theorem \ref{thm:TracialDiv} using it and Lemma \ref{lem:TrDiv}.

\begin{proof}[Proof of Theorem \ref{thm:TracialDiv}.]
Let $A,B$ be as in Theorem \ref{thm:TracialDiv}, and set $m := \mathrm{dim}_\mathrm{nuc} B$.
We first show that we can find $\phi$ for which \eqref{TracialDivEq1} holds, then show that \eqref{TracialDivEq2} follows.
For this, we wish to apply Proposition \ref{prop:MinimizingDefect}, with $Q:=Q(N,2m-1)$ as given by Lemma \ref{lem:TrDiv}.

For a $\mathrm{C}^*$-algebra $C$ as in the statement of Proposition \ref{prop:MinimizingDefect}, we have by \cite[Lemma 7.1]{UnitlessZ} that $\mathrm{dim}_\mathrm{nuc} C \leq 2m-1$.
Thus, Lemma \ref{lem:TrDiv} tells us that there exists $\psi\colon M_k(\C) \to \F(C,A)$ such that $d_\tau(1) \leq Q[\psi(1)]$, verifying the hypothesis of Proposition \ref{prop:MinimizingDefect}, and therefore $\phi$ exists satisfying \eqref{TracialDivEq1}.

Now, given that $\phi$ satisfies \eqref{TracialDivEq1}, we have $[1] \leq [\phi(1)] + [(1-\phi(1)]$, and therefore, for any $\tau \in \mathrm{QT}(\F(B,A))$,
\[ 
d_\tau(1) \leq d_\tau(\phi(1)) + d_\tau(1-\phi(1)) \leq d_\tau(\phi(1)) + \e d_\tau(1). 
\]
When $\tau$ is bounded, we may cancel to get $d_\tau(\phi(1)) \geq (1-\e)d_\tau(1)$, so that 
\[
d_\tau(\phi(e_{11})) \geq \frac{1-\e}k d_\tau(1) \geq (\frac1k-\e)d_\tau(1).\qedhere
\]
\end{proof}

Some preparation is needed before we prove Proposition \ref{prop:MinimizingDefect}.
First, we will need the following result by Winter:

\begin{lemma}\label{lem:commutingranges}
Let $D$ be a $\mathrm{C}^*$-algebra and $\phi_1,\phi_2\colon M_k(\C)\to D$ be c.p.c.\ order zero maps with ranges that commute.
\begin{enumerate}[(i)]
\item
If $\phi_1(1)+\phi_2(1)\leq 1$  then there exists a c.p.c.\ order zero map $\phi\colon M_k(\C)\to D$
such that $\phi(1)=\phi_1(1)+\phi_2(1)$.
\item
There exists a c.p.c.\ order zero map $\phi\colon M_k(\C)\to D$ such that $\phi(1)\geq \phi_1(1)$ and
\[
(1-\phi(1))=(1-\phi_1(1))(1-\phi_2(1)).
\]
\end{enumerate}
\end{lemma}
\begin{proof}
(i) This is \cite[Lemma 2.3]{Winter:Zinitial} (cf.\ also \cite[Lemma 7.6]{KirchbergRordam:CentralSeq}).

(ii) This follows from (i) applied to $\phi_1$ and $(1-\phi_1(1))\phi_2$.
\end{proof}

We record a small functional calculus maneouvre in the following lemma, that allows us to strengthen \eqref{MinimizingDefect-Input}.

\begin{lemma}\label{lem:DefectManeouvre}
Let $A$ be a unital $\mathrm{C}^*$-algebra, let $\phi\colon M_k(\C) \to A$ be a c.p.c.\ order zero map.
If $d_\tau(1) \leq Q'd_\tau(\phi(1))$ for all bounded quasitraces $\tau$ on $A$, for some $Q > 0$, then there exists a c.p.c.\ order zero map $\psi\colon M_k(\C) \to A$ such that 
\begin{equation}\label{defect}
d_\tau(1-\psi(1)) \leq  (1-\frac 1{2Q})d_\tau(1).
\end{equation}
\end{lemma}

\begin{proof}
Using Lemma \ref{lem:TrAlgebraic}, we can see that there exists $\dl > 0$ such that $d_\tau(1) \leq 2Q\cdot d_\tau((\phi(1)-\dl)_+)$ for all bounded quasitraces $\tau$.
Let us set $\psi=g_{\frac\delta 2}(\phi)$. If $\tau$ is unbounded, \eqref{defect} holds automatically.
Otherwise,
\begin{align*}
2Q\cdot d_\tau(1-\psi(1))+ d_\tau(1) &\leq   2Q\cdot d_\tau(1-\psi(1))+ 2Q\cdot d_\tau((\phi-\delta)_+)\\
&\leq 2Q\cdot d_\tau(1),
\end{align*}
and from here we can cancel $d_\tau(1)$ to get \eqref{defect}.
\end{proof}

\begin{proof}[Proof of Proposition \ref{prop:MinimizingDefect}]
Let us set $\beta$ to be the infimum of $\e > 0$ for which there exists $\phi\colon M_k(\C) \to \F(B,A)$ such that $[1-\phi(1)] <_s \epsilon\cdot [1]$.
It is clear by the hypothesis that $\beta \leq 1$.
We must show that $\beta = 0$, and to do this, we shall show that $\beta$ satisfies
\[ 
\beta \leq \Big(1-\frac1{3Q}\Big)\beta. 
\]
Let $\e > \beta$ and let  $\psi\colon M_k(\C) \to \F(B,A)$ be such that $d_\tau(1-\psi(1)) \leq \epsilon\cdot  d_\tau(1)$. Let us lift $\psi$ to a c.p.c.\ order zero map $\hat\psi\colon M_k(\C) \to A_\omega \cap B'$ and set $C := C^*(B \cup \hat\psi(M_k(\C)))$.
By the hypothesis and Lemma \ref{lem:DefectManeouvre}, there exists $\phi_0\colon M_k(\C)\to \F(C,A)$ such that
\begin{equation}
\label{eq:FullCpcSmallComplementPhi0}
d_\tau(1-\phi_0(1)) \leq (1-\frac{1}{2Q})d_\tau(1)
\end{equation}
for all quasitraces $\tau$.
Notice that $\F(C,A) \equiv \F(B,A) \cap \psi(M_k(\C))'$, so that we can view $\phi_0$ and $\psi$ as c.p.c.\ order zero maps into $\F(B,A)$ with commuting ranges.

By Lemma \ref{lem:commutingranges}, there exists $\phi\colon M_k(\C)\to \F(B,A)$
such that $(1-\phi(1))=(1-\psi(1))(1-\phi_0(1))$.
For any bounded trace $\tau$ on $\F(B,A)$, using Lemma \ref{lem:LftFacts} and \eqref{eq:FullCpcSmallComplementPhi0}, we have
\begin{align*}
d_\tau(1-\phi(1)) &= d_\tau((1-\psi(1))(1-\phi_0(1))) \\
&\leq (1-\frac1{2Q})d_\tau(1-\psi(1)) \\
&\leq  (1-\frac1{2Q})\e\cdot d_\tau(1).
\end{align*}
Hence, $[1-\phi(1)]<_s (1-\frac{1}{3Q})\e[1]$. This shows that $\beta\leq (1-\frac{1}{3Q})\e$ for any $\e>\beta$, and so $\beta \leq (1-\frac1{3Q})\beta$. 
\end{proof}

\section{$\mathcal Z$-stability}
\label{sec:ZStability}

This section contains the proofs of conjectures (C1) and (C2) in various cases.

\subsection{The simple case}
\label{sec:SimpleZ}

Here, we give a simplified proof of the main results of \cite{Winter:pure} and \cite{UnitlessZ}:

\begin{theorem}
(\cite[Theorem 7.1]{Winter:pure}, \cite[Theorem 8.5]{UnitlessZ})
\label{thm:SimpleZ}
Let $A$ be a simple, separable, stably finite $\mathrm{C}^*$-algebra which is $(M,N)$-pure, for some $M,N \in \N$, and which has locally finite nuclear dimension.
Then $A$ is $\mathcal Z$-stable.
\end{theorem}

The key step is going from the conclusion of Theorem \ref{thm:TracialDiv} to a nontracial version, which requires a certain finiteness condition on $\F(B,A)$, namely that the unit is not stably properly infinite in any quotient.
We shall see that this finiteness condition holds when $A$ is simple and tracial, has $M$-comparison for some $M$ and $B \subseteq A$ has finite nuclear dimension.
Interestingly, even if $A$ is an infinite UHF algebra, if $B \subseteq A_\omega$ (instead of $\subseteq A$) with finite nuclear dimension, $\F(B,A)$ may have purely infinite quotients, as shown in Example \ref{ex:NoCancel}. At the end of this subsection we given a separate argument that  deals with the simple purely infinite case. 

\begin{proposition}
\label{prop:FinitenessZ}
Let $A$ be a $\mathrm{C}^*$-algebra and let $B \subset A_\omega$ be a separable $\mathrm{C}^*$-subalgebra.
Suppose that $A$ is $(M,N)$-pure, that $\mathrm{dim}_\mathrm{nuc} B < \infty$, and that $1$ is not stably properly infinite in any quotient of $\F(B,A)$.
Then there exists a unital embedding of $\mathcal Z$ into $\F(B,A)$.
\end{proposition}

\begin{proof}
We shall show that, for each $k \in \N$, there exists a unital *-homomorphism from $\mathcal Z_{k,k+1}$ to $\F(B,A)$. By a diagonal sequence argument this implies that $\mathcal Z$ embeds unitally in $\F(B,A)$ (see  \cite[Proposition 5.3]{UnitlessZ} for the case of $\F(A)$). 

Let us fix $k \in \N$.
By Proposition \ref{prop:Fcomparison}, there exists $\overline{M} \in \N$ such that $\F(B,A)$ has $\overline M$-comparison.
By Theorem \ref{thm:TracialDiv}, there exists a c.p.c.\ map $\phi\colon M_{k(\overline M+1)}(\C) \to \F(B,A)$ of order zero such that 
$d_\tau(1-\phi(1)) < \frac 1{2(\overline M+1)k+1} d_\tau(1)$ and 
$\frac 1{2(\overline M+1)k} d_\tau(1) < d_\tau(\phi(e_{11}))$ for every bounded quasitrace $\tau$ on $\F(B,A)$.
By Lemma \ref{lem:TrAlgebraic}, we have
\[ L[1] + p[1] \leq L[1] + q[\phi(e_{11})] \]
for some $L,p,q \in \N$ with $\frac pq > \frac 1{2(\overline M+1)k}$.
Let $\e>0$ be such that $L[1] + p[1] \leq L[1] + q[(\phi(e_{11})-\e)_+]$.
This implies that $1$ is not stably properly infinite modulo the ideal generated by $(\phi(e_{11})-\e)_+$. So by the hypothesis, $(\phi(e_{11})-\e)_+$ is full.

For every bounded quasitrace $\tau$ on $\F(B,A)$  we have 
\[
d_\tau(1-\phi(1)) < \frac 1{2(\overline M+1)k+1} d_\tau(1) < \gamma d_\tau((\phi(e_{11})-\e)_+)
\] 
for some $\gamma < 1$. On the other hand, if $\tau$ is unbounded then 
\[
d_\tau(1-\phi(1)) \leq \infty = d_\tau((\phi(e_{11})-\e)_+).
\]
Thus, $[1-\phi(1)] <_s [(\phi(e_{11})-\e)_+]$, and therefore by $\overline M$-comparison in $\F(B,A)$ we have $[1-\phi(1)] \leq (\overline M+1)[(\phi(e_{11})-\e)_+]$.
Let us now view $M_{k(\overline M+1)}(\C)$ as $M_k(\C) \otimes M_{\overline M+1}(\C)$, and set $\psi:= \phi(\,\cdot \otimes 1_{\overline M+1})\colon M_k(\C) \to \F(B,A)$.
We can restate our latest conclusion as $[1-\psi(1)] \leq [(\psi(e_{11})-\e)_+]$.
By \cite[Proposition 5.1]{RordamWinter:Z}, it follows that there is a unital *-homomorphism $\mathcal Z_{k,k+1} \to \F(B,A)$, as required.
\end{proof}

Now, we verify the above finiteness condition. We will need some lemmas that will be reused in the sequel.
Let us say that a Cuntz class $[c]$ is \demph{pseudocompact} if $[c]\propto [(c-\epsilon)_+]$
for some $\epsilon>0$. If $[c]$ is pseudocompact then $\mathrm{Ideal}(c)$
is a compact open set of $\mathrm{Prim}(A)$. Conversely, if $\mathrm{Ideal}(c)$ is compact then $[(c-t)_+]$ is pseudocompact for all sufficiently small $t>0$.

Recall that a Cuntz class $[c]$ is said to be stably properly infinite if it is nonzero and $(n+1)[c]=n[c]$
for some $n\in \N$. 

\begin{lemma}\label{lem:Finitepseudocompacts}
Let $A$ be a $\mathrm{C}^*$-algebra with $M$-comparison and such  that no nonzero simple subquotient of $A$ is purely infinite. 

(i) Then no quotient of $A$ contains a pseudocompact, stably properly infinite element.
 
(ii) If $[c]$ is pseudocompact and 
$L[c] + p[c] \leq L[c]+q[b]$ for some $[b]$ and $p,q>0$ then $[c]\leq (M+1)k[b]$
for any $k > \frac qp$ .
\end{lemma}

\begin{proof}
(i)
It suffices to show that $\Cu(A)$ contains no pseudocompact, stably properly infinite element.
Assume for a contradiction that $[c]$ is pseudocompact and stably properly infinite.
Then a sufficiently large multiple of $[c]$ is compact and properly infinite.
Let $J$ be a maximal ideal not containing $c$. Then $\mathrm{Ideal}(c)/J$ is simple, has $M$-comparison, and a sufficiently large multiple of $[\pi_J(c)]$ is compact and  properly infinite. It follows that $\mathrm{Ideal}(c)/J$ is a purely infinite $\mathrm{C}^*$-algebra (see the proof of Proposition \ref{prop:nucleardichotomy}), which contradicts our hypotheses.

(ii) 
Let $I=\mathrm{Ideal}(b)$. Passing to the quotient by $I$ we get $(L+p)[\pi_I(c)]=L[\pi_I(c)]$.
Since $[\pi_I(c)]$ cannot be stably properly infinite, it must be 0. That is, $c$ belongs to the ideal generated by $b$.
For any $\tau \in \mathrm{QT}(A)$, if $d_\tau(b) <\infty$ then 
$d_\tau((c-t)_+)<\infty$ for all $t>0$ and so $d_\tau(c)<\infty$ by the pseudocompactness of $[c]$.
Hence, from the relation $L[c]+p[c]\leq L[c]+q[b]$ we get that $d_\tau(c)\leq \frac qp d_\tau(b)$ for all $\tau\in \mathrm{QT}(A)$ and so $[c] <_s k[b]$.
Since $A$ has $M$-comparison, we conclude that $[c] \leq (M+1)k[b]$.
\end{proof}

\begin{lemma}\label{lem:Aomegaquotients}
Suppose that $A$ has $M$-comparison.  If   no quotient of $A$ contains a stably properly infinite compact element then the same is true for $\prod_{i=0}^\infty A$ and $A_\omega$.
\end{lemma}

\begin{proof}
This property clearly passes to quotients, so we prove it just for $\prod A$.
That $M$-comparison passes is shown in Proposition \ref{prop:DivCompSequence}.
Suppose that $[a] \in \Cu(\prod A)$ becomes stably properly infinite and compact in some quotient.
This means that
\[
(n+1)[a]\leq n[(a-\overline \epsilon)_+]+[b]
\] 
in $\Cu(\prod A)$, where $a$ is not in the ideal generated by $b$.
Let $\epsilon>0$ and find $\delta>0$
such that
\[
(n+1)[(a-\epsilon)_+]\leq n[(a-\overline \epsilon)_+]+[(b-\delta)_+].
\]
Then  for each $i$ we have
\[
(n+1)[(a_i-\epsilon)_+]\leq n[(a_i-\overline \epsilon)_+]+[(b_i-\delta)_+].
\]
Let us assume without loss of generality that $\epsilon<\overline \epsilon$.
Then, with $I:=\mathrm{Ideal}((b_i-\dl)_+)$, we have
\[ (n+1)[(\pi_I(a_i-\e)_+] \leq n[\pi_I((a_i-\overline \epsilon)_+], \]
so that by Lemma \ref{lem:Finitepseudocompacts} (i) we must have $(a_i-\e)_+ \in \mathrm{Ideal}((b_i-\dl)_+)$.
Arguing as in the proof of Lemma \ref{lem:Finitepseudocompacts} (ii), we get that $d_\tau((a_i-\e)_+) \leq d_\tau((b_i-\dl)_+)$ for all quasitraces $\tau\in \mathrm{QT}(A)$ (it suffices to consider those $\tau$ for which $d_\tau((b_i-\dl)_+)<\infty$). Thus, $[(a_i-\e)_+]<_s 2[(b_i-\dl)_+]$
and by $M$-comparison, $[(a_i-\e)_+] \leq 2(M+1)[(b_i-\dl)_+]$ for all $i$.
It follows that $[(a-\e)_+] \leq 2(M+1)[b]$.
Since $\epsilon>0$ is arbitrary, we get that $[a]\leq 2(M+1)[b]$, and in particular $a$ belongs to the ideal generated by $b$. This is a contradiction.
\end{proof}

\begin{proposition}
\label{prop:SimpleFiniteness}
Let $A$ be a simple tracial $\mathrm{C}^*$-algebra with $M$-comparison for some $M$.
Let $B \subseteq A$ be a $\mathrm{C}^*$-subalgebra of nuclear dimension at most $m$.
Then in every quotient of $\F(B,A)$, $1$ is not stably properly infinite.
\end{proposition}

\begin{proof}
For a contradiction, suppose that in some nonzero quotient of $\F(B,A)$, we have $(k+1)[1] \leq k[1]$.
Equivalently, there exists a non-full element $[b] \in \Cu(\F(B,A))$ such that $(k+1)[1] \leq k[1] + [b]$.
Consequently, for some $\e>0$, $(k+1)[1] \leq k[1]+[(b-\e)_+]$.

For $c \in B_+$ and $\eta>0$, we shall show that $[(c-\eta)_+] \leq (M+1)[(b-\e)_+(c-\eta)_+]$ in $\Cu(A_\omega)$.
We know that $(k+1)[(c-\eta)_+] \leq k[(c-\eta)_+] + [(b-\e)_+(c-\eta)_+]$.
Since $A$ is simple, $(c-\eta)_+$ is pseudocompact, so that by Lemma \ref{lem:Aomegaquotients} and Lemma 
\ref{lem:Finitepseudocompacts} (ii),
\[ [(c-\eta)_+] \leq (M+1)[(b-\e)_+(c-\eta)_+], \]
as required.

It now follows by Proposition \ref{prop:CommComparison} that
\[ [1] \leq 2(m+1)(M+1)[b], \]
which is a contradiction, since $[b]$ is not full.
\end{proof}

\begin{corollary}
\label{cor:SimpleZEmbed}
Let $A$ be a simple tracial $\mathrm{C}^*$-algebra which is $(M,N)$-pure for some $M,N$.
Let $B \subseteq A$ be a separable $\mathrm{C}^*$-subalgebra of finite nuclear dimension.
Then $\mathcal Z$ embeds into $F(A,B)$.
\end{corollary}

\begin{proof}
This is a direct consequence of Propositions \ref{prop:FinitenessZ} and \ref{prop:SimpleFiniteness}.
\end{proof}

\begin{proof}[Proof of Theorem \ref{thm:SimpleZ}.]
Since $A$ is separable and has locally finite nuclear dimension, Corollary \ref{cor:SimpleZEmbed} and a diagonal sequence argument implies that $\ZZ$ embeds unitally in $\F(A)$. It follows by Proposition \ref{prop:Zcriterion} that $A$ is $\mathcal Z$-stable.  
\end{proof}

Here is an example to show that the conclusion of Proposition \ref{prop:SimpleFiniteness} fails if we allow $B$ to be positioned in $A_\omega$ instead of in $A$.

\begin{example}
\label{ex:NoCancel}
Let $A$ be an infinite dimensional UHF algebra.
By \cite{Kirchberg:nonsemisplit,Voiculescu:Quasidiagonal}, $C_0((0,1]) \otimes \mathcal O_2$ is quasidiagonal, and therefore there exists an embedding
\[ \phi\colon C_0((0,1]) \otimes \mathcal O_2 \to A_\omega. \]
Set $B := \phi(C_0((0,1]) \otimes 1_{\mathcal O_2}) \subseteq A_\omega$.
Since $B$ is commutative, $\F(B,A) = A_\omega \cap B'$ contains $B$.
By unitizing, we see that $\F(B,A)$ is a $C([0,1])$-algebra.

Let us see that the quotient of $\F(B,A)$ given by the fibre at $1$ is infinite.
Surely, it is clear that $\phi(C_0((0,1]) \otimes \mathcal O_2)) \subseteq \F(B,A)$.
Therefore, the fibre at $1$ contains a copy of $\mathcal O_2$, which implies that it is infinite.
What is more, we may pick a simple quotient of this fibre (which is of course a quotient of $\F(B,A)$), and it will have $3$-comparison by Proposition \ref{prop:Fcomparison}, which implies that it is purely infinite.
\end{example}

If a simple $\mathrm{C}^*$-algebra is traceless and has $M$-comparison then it is purely infinite (see the proof of Proposition \ref{prop:nucleardichotomy}). If in addition the $\mathrm{C}^*$-algebra is nuclear, then it is $\mathcal O_\infty$-stable by Kirchberg's theorem and a fortiori also $\ZZ$-stable.  Below, we give an independent proof of $\ZZ$-stability for  simple separable purely infinite $\mathrm{C}^*$-algebras
with locally finite nuclear dimension.

\begin{proposition}
Let $A$ be a $\mathrm{C}^*$-algebra that is separable, unital, simple, purely infinite, and of locally finite nuclear dimension. Then $A$ is $\ZZ$-stable.
\end{proposition}

\begin{proof}
Since $A$ is simple and separable, we have  $\F(A)\neq \C$ by \cite[Lemma 2.8]{Kirchberg:CentralSequences}. 
Thus, there exist non-zero orthogonal positive elements $d^0,d^1\in \F(A)$. Let us choose 
$0<\delta<\|d^0\|,\|d^1\|$. Fix $i=0,1$ and consider the set
$\{c\in A\mid c(d^i-\delta)_+=0\}$. This is a closed two-sided ideal of $A$. Hence, it is either $\{0\}$
or $A$. It cannot be the latter, since $1$ is not in it. Thus, $c(d^0-\delta)_+\neq 0$
for all non-zero $c\in A$. Since $A_\omega$ is simple and purely infinite, 
$c\precsim c(d^i-\delta)_+$ in $A_\omega$, for all positive $c\in A$. Applied to $(c-\epsilon)_+$
for a fixed $\epsilon>0$ we get $(c-\epsilon)_+\precsim (c-\epsilon)_+(d^i-\delta)_+$. Thus,
$1\precsim_A d^i$ for $i=0,1$. 

Let $B\subset A_\omega$ be a separable $\mathrm{C}^*$-subalgebra. By a standard argument passing to subsequences applied to $d^0$  
and $d^1$, we can find positive orthogonal elements $\tilde d^0,\tilde d^1\in \F(B,A)$ such that $1\precsim_B \tilde d^i$
for $i=0,1$. Suppose that the nuclear dimension of $B$ is at most $m$. Then by Proposition \ref{prop:Fcomparison} we have $[1]\leq (2m+1)[\tilde d^i]$ in the Cuntz semigroup of $\F(B,A)$ for $i=0,1$. Thus, by Theorem \ref{thm:OrthogZ}, $A$ is $\ZZ$-stable.   
\end{proof}

\subsection{Full orthogonal elements in $\F(A)$}

\begin{theorem}
\label{thm:OrthogZ}
Let $A$ be a separable $\mathrm{C}^*$-algebra which is $(M,N)$-pure, for some $M,N \in \N$, and which has locally finite nuclear dimension.
Suppose that for each $m \in \N$, there exist $P_m \in \N$ such that the following holds:
If $B \subset A_\omega$ is a separable $\mathrm{C}^*$-subalgebra with nuclear dimension at most $m$, then there exist orthogonal elements $d_0,d_1 \in \F(B,A)_+$ such that $[1] \leq P_m[d_i]$ for $i=0,1$.
Then $A$ is $\mathcal Z$-stable.
\end{theorem}

\begin{remark*}
(i) In particular, the above result applies when $A$ is separable $\mathrm{C}^*$-algebra, $(M,N)$-pure, has locally finite nuclear dimension, and $\F(A)$ contains two orthogonal full elements.

(ii) The above theorem has a strong converse: if $A$ is $\ZZ$-stable then it is $(0,0)$-pure, and $\F(B,A)_+$ has orthogonal elements $d_0,d_1$ which satisfy $[1] \leq 3[d_i]$ for $i=0,1$.
Certainly, $(0,0)$-pureness is shown (essentially) in \cite[Proposition 3.7]{Winter:pure} (primarily using \cite{Rordam:Z}).
Also, it is well-known that $\ZZ$ contains orthogonal elements $\hat d_0, \hat d_1$ such that $[1] \leq 3[(\hat d_i - \e)_+]$, for $i=0,1$, and some $\e > 0$.
Viewing $A$ as $A \otimes \ZZ \otimes \ZZ \otimes \cdots$, we can easily use these to produce orthogonal elements $d_0,d_1 \in \F(A)_+$ such that $[1] \leq 3[d_i]$.
For $\F(B,A)$, we simply use a speeding-up argument, as in the proof of \cite[Proposition 4.4]{Winter:pure}.

(iii) Theorem \ref{thm:OrthogZsimple} is an obvious special case to the above theorem and (ii).
\end{remark*}

\begin{proof}
This proof contracts ideas found in the proofs of Theorems \ref{thm:TracialDiv} and \ref{thm:SimpleZ}.
We must show that for each $B$ of finite nuclear dimension and each $k \in \N$, there exists a unital *-homomorphism $\mathcal Z_{k,k+1} \to \F(B,A)$.

Using the idea behind the proof of Proposition \ref{prop:ManyOrthogTrFull}, we see that there exists $Q_m \in \N$ such that, if $B \subseteq A_\omega$ is a separable $\mathrm{C}^*$-subalgebra of nuclear dimension at most $m$, then there exist orthogonal elements $d_0,\dots,d_{2m+1} \in \F(B,A)_+$ such that $[1] \leq Q_m[d_i]$.

Fixing $B \subseteq A_\omega$ of nuclear dimension at most $m < \infty$, let us show that the hypothesis of Proposition \ref{prop:MinimizingDefect} holds, with $Q:=Q_{2m-1}\cdot N$.
Let $C$ be a $\mathrm{C}^*$-algebra as in the statement of Proposition \ref{prop:MinimizingDefect}; it has nuclear dimension at most $\overline m := 2m-1$.
Therefore, let $d_0,\dots,d_{2\overline m +1} \in \F(C,A)_+$ be orthogonal, such that $[1] \leq Q_{\overline m}[d_i]$.
Using this with Proposition \ref{prop:VeryWeakDiv}, we get c.p.c.\ order zero maps $\phi_0,\dots,\phi_{2\overline m+1}\colon M_k(\C) \to \F(C,A)$ such that $\phi_i(M_k(\C)) \subseteq \her(d_i)$ for each $i$ and $[1] \leq Q_{\overline m}N \sum_{i=0}^{2m+1} [\phi_i(1_k)]$.
Since the $d_i$'s are pairwise orthogonal, it follows that $\phi := \sum_{i=0}^{2m+1} \phi_i$ is a c.p.c.\ order zero map, and we see that $[1] \leq Q [\phi(1)]$, as required.

By Proposition \ref{prop:MinimizingDefect}, for any $\e > 0$, we may find a c.p.c.\ order zero map $\phi_1\colon M_k(\C) \to \F(B,A)$ such that $[1-\phi_1(1)] <_s \e[1]$.
By the argument above, we may then find another c.p.c.\ order zero map $\phi_2\colon M_k(\C) \to \F(B,A) \cap \phi_0(M_k(\C))'$ such that $[1] \leq Q[\phi_2(1)]$.
Then, we may effectively combine these two order zero maps by Lemma \ref{lem:commutingranges} (ii), to get a c.p.c.\ order zero map $\phi\colon M_k(\C) \to \F(B,A)$ such that $[1] \leq Q[\phi(1)]$ and $[1-\phi(1)] <_s \e[1]$.

By Proposition \ref{prop:Fcomparison}, $\F(B,A)$ has $\overline{M}$-comparison for some $\overline M \in \N$.
Now, given $k \in \N$, as explained in the previous paragraph, we may find a c.p.c.\ order zero map $\phi\colon M_{k(\overline M+1)}(\C) \to \F(B,A)$ such that $[1] \leq Q[\phi(1)]$ and $[1-\phi(1)] <_s \frac1{k(\overline M+1)Q} [1]$.
Let $\e > 0$ be such that $[1] \leq Q[(\phi(1)-\e)_+]$.
Combining these, we see that, $[1-\phi(1)] <_s [(\phi(e_{11})-\e)_+]$, so that by $\overline M$-comparison, $[1-\phi(1)] \leq (\overline M+1)[(\phi(e_{11})-\e)_+]$.

Let us now view $M_{k(\overline M+1)}(\C)$ as $M_k(\C) \otimes M_{\overline M+1}(\C)$, and set $\psi:= \phi(\,\cdot \otimes 1_{\overline M+1})\colon M_k(\C) \to \F(B,A)$.
We can restate our latest conclusion as $[1-\psi(1)] \leq [(\psi(e_{11})-\e)_+]$.
By \cite[Proposition 5.1]{RordamWinter:Z}, it follows that there is a unital *-homomorphism $\mathcal Z_{k,k+1} \to \F(B,A)$, as required.
\end{proof}

\subsection{$\mathrm{C}^*$-algebras with finite subquotients and a basis of compact-open sets for the spectrum}
\label{sec:AlgLatt}

Here, we show the following:

\begin{theorem}
\label{thm:AlgLattZ}
Let $A$ be a separable $\mathrm{C}^*$-algebra of locally finite nuclear dimension which is $(M,N)$-pure for some $M,N>0$. Suppose also that
\begin{enumerate}
\item no nonzero simple subquotient of $A$ is purely infinite; and
\item $\mathrm{Prim}(A)$ has a basis of compact open sets.
\end{enumerate}
Then $A$ is $\ZZ$-stable.
\end{theorem}
Combined with Theorem \ref{thm:purenonelementary}, the preceding theorem yields at once the following
\begin{corollary}
Let $A$ be a separable $\mathrm{C}^*$-algebra of  finite nuclear dimension and with no elementary subquotients. Suppose also that $A$ satisfies conditions (i) and (ii) of the previous theorem. Then $A$ is $\ZZ$-stable.
\end{corollary}

Before getting towards the proof of Theorem \ref{thm:AlgLattZ}, let us point out that the conditions (i) and (ii) of that theorem, together with having finite nuclear dimension,  hold in the following cases:
\begin{enumerate}
\item[(a)] If $A$ has finite decomposition rank and the ideal property (as defined in \cite[Definition 1.5.2]{Rordam:ClassBook}).
In particular, this is the case if $A$ has finite decomposition rank and real rank zero.
\item[(b)] If $A = C(X) \rtimes_\alpha \Z^n$, where $X$ is the Cantor set and 
$\alpha\colon\Z^n \to \mathrm{Aut}(A)$ is a free action.
Indeed, Szab\'o has shown in \cite{Szabo:Rokhdim} that such crossed products have finite nuclear dimension.
Since the action is free, it is not hard to see that every ideal of the crossed product is generated by an ideal of $C(X)$ (cf.\ \cite{Sierakowski:Ideals} for example); since $C(X)$ has the ideal property, it follows that $A$ does as well.
\end{enumerate}
We note that for $\mathrm{C}^*$-algebras of the form in (b), the condition of no elementary subquotients  (which is of course necessary for $\mathcal Z$-stability) is equivalent to the following: there is no pair of $\alpha$-invariant closed subsets $Y,Z \subseteq X$ such that $Y \setminus Z$ is nonempty and (at most) countable.

Let us prepare now to prove Theorem \ref{thm:AlgLattZ}, which will be done by applying Theorem \ref{thm:OrthogZ}.
The following  lemma clarifies the role of the condition (ii) in Theorem \ref{thm:AlgLattZ}.

\begin{lemma}
\label{lem:DensePseudocompact}
Let $A$ be a $\mathrm{C}^*$-algebra such that  the topology of $\mathrm{Prim}(A)$ has a basis of compact open sets. The for each $\epsilon>0$ the set of elements $c\in A$ such that $(c-\epsilon)_+$ is pseudocompact
is   dense in $A_+$.
\end{lemma}

\begin{proof}
Let $a\in A_+$ and $\epsilon>0$. Let $I=\mathrm{Ideal}((a-\epsilon)_+)$.
Let us write $I$ as supremum of compact ideals. Since the sum of compact ideals is compact, we can assume that this supremum is upward directed. Let $(I_\lambda)$ be increasing with supremum $I$. Then  $\mathrm{her}((a-\epsilon)_+)=\overline{\bigcup_\lambda I_\lambda} \cap \mathrm{her}((a-\epsilon)_+)$. Let us find $b\in \mathrm{her}(a-\epsilon)_+\cap I_\lambda$ that is close to $(a-\epsilon)_+$. Let us assume that $b$ generates $I_\lambda$. Now define $c=a-(a-\epsilon)_++b$. Then $(c-\epsilon)_+=b$ and $c$ is close to $a$. 
\end{proof}

\begin{lemma}
Let $A$ be a $\mathrm{C}^*$-algebra of locally finite nuclear dimension. Suppose that   the topology of $\mathrm{Prim}(A)$ has a basis of compact open sets. Let $F\subset A_+$ be a finite set of contractions, and let $\epsilon, \gamma>0$.
Then there exists
a $\mathrm{C}^*$-subalgebra $B\subseteq A$ of finite nuclear dimension such that for each $c\in F$ there exists $c'\in B_+$ such that $c\approx_\gamma c'$ and $[(c'-\epsilon)_+]$ is pseudocompact.
\end{lemma}

\begin{proof}
By the previous lemma, we can find a finite set $F'$ such that  $F\subseteq_{\frac\gamma2} F'$
and for each $c\in F'$ we have that $[(c-\epsilon)_+]$ is a pseudocompact element of $\Cu(A)$. 
Let $t_0 \in (0,\frac\gamma2)$ be such that, for each $c'\in F'$ we have that $[(c'-\epsilon)_+]\propto [(c'-\epsilon-t_0)_+]$.
If $c''\in A_+$ is such that $c'' \approx_{\frac{t_0}3} c'$ then 
\[
[(c''-\epsilon-\frac{t_0}3)_+]\leq [(c'-\epsilon)_+]\propto [(c'-\epsilon-t_0)_+]\leq [(c''-\epsilon-\frac{2t_0}3)_+].
\]
Thus, $(c''-\epsilon-\frac{t_0}3)_+$
is pseudocompact. Let us find $B\subseteq A$,
of finite nuclear dimension, such that for each $c'\in F'$ there exists $c''\in B$
such that $c'' \approx_{\frac{t_0}3} c'$. Set $c'''=(c''-\frac{t_0}3)_+\in B$. Then $[(c'''-\epsilon)]_+$
is pseudocompact and $c''' \approx_\gamma c$.
\end{proof}

\begin{lemma}
\label{lem:AlgLattOrthog}
Let $A$ be as in Theorem \ref{thm:AlgLattZ}.
Then there exist orthogonal positive elements $d^0,d^1\in \F(A)$ such that $1 \preceq_A d^i \otimes 1_{3(M+1)}$ for $i=0,1$ (where $\preceq_A$ is as defined after Lemma \ref{lem:precBChar}).
\end{lemma}

\begin{proof}
By the remark following Lemma \ref{lem:precBChar}, it suffices to show that there exist orthogonal elements $d^0,d^1 \in \F(A)$ and $\dl > 0$ such that, for each contraction $c \in A_+$,
\[
[(c-\frac12)_+]\leq 3(M+1)[(cd^i-\delta)_+].
\]
Our $\dl$ will be $\frac16$.

It suffices by a diagonal sequence argument to show that for each finite set $F\subset A_+$ of positive contractions and $\gamma>0$ there exist $d^0,d^1\in A_\omega$
such that $\|[d^i,F]\|<\gamma$ and 
\[
[(c-\frac12)_+]\leq 3(M+1)[((d^i)^{1/2}c(d^i)^{1/2}-\delta)_+]
\]
for all $c\in F$.
By the previous lemma, there exists $B\subseteq A$ of finite nuclear dimension such that for each $c\in F$ there exists $c'\in B$ such that $c \approx_{\frac16} c'$ and $[(c'-\frac13)_+]$ is pseudocompact.

By Theorem \ref{thm:TracialDiv}, there exist two orthogonal elements $d^0,d^1 \in \F(B,A)$ such that $d_\tau(d^i) > \frac13 d_\tau(1)$ for all bounded quasitraces $\tau$ on $\F(B,A)$.
By Lemma \ref{lem:TrAlgebraic}, there exist $L,p,q \in \N$ such that $\frac pq > \frac13$ and $L[1]+p[1] \leq L[1]+q[d^i]$ in $\Cu(\F(B,A))$ for $i=0,1$.
(Although it is unimportant to the argument here, the proof of Lemma \ref{lem:TrDiv} shows why we may use the same values for both $d^0$ and $d^1$.)
Let $\e > 0$ be such that $L[1]+p[1] \leq L[1]+q[(d^i-\e)_+]$ for $i=0,1$; without loss of generality (by possibly modifying $d^i$ by functional calculus), we may assume that $\e = \frac13$.

It follows that $L[b]+p[b] \leq L[b] + q[(d^i-\frac13)_+b]$ in $\Cu(A_\omega)$, for all $b \in B_+$.
When $[b]$ is pseudocompact, we get that
$[b] \leq 3(M+1)[(d^i-\frac13)_+b]$, by Proposition \ref{prop:DivCompSequence} and Lemma \ref{lem:Finitepseudocompacts} (ii).

Thus, for $c \in F$,
\begin{align*}
[(c-\frac12)_+] &\leq [(c'-\frac13)_+] \\
&\leq 3(M+1)[(d^i-\frac13)_+(c'-\frac13)_+] \\
&\leq 3(M+1)[((d^i)^{1/2}c'(d^i)^{1/2}-\frac13)_+] \\
&\leq 3(M+1)[((d^i)^{1/2}c(d^i)^{1/2}-\frac16)_+],
\end{align*}
where on the third line, we used Lemma \ref{lem:CutdownMult}, and on the last line we used the fact that $(d^i)^{1/2}c'(d^i)^{1/2} \approx_{\frac16} (d^i)^{1/2}c(d^i)^{1/2}$.
\end{proof}

\begin{proof}[Proof of Theorem \ref{thm:AlgLattZ}]
Let $B \subseteq A_\omega$ be a separable $\mathrm{C}^*$-subalgebra of nuclear dimension at most $m$.
By Lemma \ref{lem:AlgLattOrthog} and a speeding-up argument (such as in the proof of \cite[Proposition 4.4]{Winter:pure}), there exist orthogonal full elements $d^0,d^1 \in \F(B,A)$ such that $1 \preceq_B d^i \otimes 1_{3(M+1)}$ for $i=0,1$.
Thus, by Proposition \ref{prop:CommComparison},
\[ 
[1] \leq 3(M+1)(2m+2)[d^i] 
\]
in $\Cu(\F(B,A))$, for $i=0,1$.

Thus, the hypothesis of Theorem \ref{thm:OrthogZ} is satisfied  with $P_m:= 6(M+1)(m+1)$; consequently, $A$ is $\mathcal Z$-stable.
\end{proof}

\subsection{$\mathrm{C}^*$-algebras with Hausdorff spectrum and finite quotients}

Here we show the following:

\begin{theorem}
\label{thm:HausdorffZ}
Let $A$ be a separable $\mathrm{C}^*$-algebra of locally finite nulear dimension which is  $(M,N)$-pure for some $M,N\in \N$. Suppose also that
\begin{enumerate}
\item no nonzero simple quotient of $A$ is purely infinite; and
\item the primitive ideal space of $A$ is Hausdorff.
\end{enumerate}
Then $A$ is $\mathcal Z$-stable.
\end{theorem}
Combined with Theorem \ref{thm:purenonelementary}, the preceding theorem yields at once the following
\begin{corollary}\label{cor:HausdorffZ}
Let $A$ be a separable $\mathrm{C}^*$-algebra of  finite nuclear dimension with no type I quotients. Suppose also that $A$ satisfies conditions (i) and (ii) of the previous theorem. Then $A$ is $\mathcal Z$-stable.
\end{corollary}

Theorem \ref{thm:HausdorffZ} may be proven by a slight adjustment to the proof in the simple case. We must generalize Proposition \ref{prop:SimpleFiniteness} as follows.

\begin{proposition}
\label{prop:HausdorffFiniteness}
Let $A$ be a $\mathrm{C}^*$-algebra with Hausdorff primitive ideal space, such that no nonzero simple quotient of $A$ is purely infinite, and suppose that $A$ has $M$-comparison for some $M$.
Let $B \subseteq A$ be a $\mathrm{C}^*$-subalgebra of nuclear dimension at most $m$.
Then in every quotient of $\F(B,A)$, $1$ is not stably properly infinite.
\end{proposition}

\begin{proof}
This proof is an adaptation of the proof of Proposition \ref{prop:SimpleFiniteness}.
For a contradiction, suppose that in some nonzero quotient of $\F(B,A)$, we have $(k+1)[1] \leq k[1]$.
Equivalently, there exists a non-full element $[b] \in \Cu(\F(B,A))$ such that $(k+1)[1] \leq k[1] + [b]$.
Consequently, for some $\e>0$, $(k+1)[1] \leq k[1]+[(b-\e)_+]$.

For $c \in B_+$ and $\eta>0$, 
we wish to show that $[(c-\eta)_+] \leq (M+1)[(b-\e)_+(c-\frac\eta2)_+]$ in $\Cu(A_\omega)$.
We know that 
\begin{align}\label{CancelPseudoc}
(k+1)[(c-\frac\eta2)_+] \leq k[(c-\frac\eta2)_+] + [(b-\e)_+(c-\frac\eta2)_+]
\end{align}
in $\Cu(A_\omega)$.

Set $X := \mathrm{Prim}(A)$ and let us regard $A$ as a $C_0(X)$-algebra in the natural way.
Set $K:=\{x\in X\mid \|c(x)\| \geq \eta\}$.
Then $(c-\frac\eta2)_+$ is a pseudocompact element of $\Cu(A_K)$, and consequently it is also pseudocompact in $\Cu((A_K)_\omega)$.
By Lemma \ref{lem:Aomegaquotients}, and Lemma \ref{lem:Finitepseudocompacts} (ii) applied to \eqref{CancelPseudoc}, we have
\[ [(c-\frac\eta2)_+] \leq (M+1)[(b-\e)_+(c-\frac\eta2)_+] \]
in $\Cu((A_K)_\omega)$, and therefore, $[(c-\eta)_+] \leq (M+1)[(b-\e)_+(c-\frac\eta2)_+]$ in $\Cu((A_K)_\omega)$.
Since the quotient map $A_\omega \to (A_K)_\omega$ is an isomorphism on $\mathrm{Ideal}((c-\eta)_+)$, it follows that $[(c-\eta)_+] \leq (M+1)[(b-\e)_+(c-\frac\eta2)_+]$ in $\Cu(A_\omega)$, as desired.

It now follows by Proposition \ref{prop:CommComparison} that
\[ [1] \leq 2(m+1)(M+1)[b], \]
which is a contradiction, since $[b]$ is not full.
\end{proof}

\begin{proof}[Proof of Theorem \ref{thm:HausdorffZ}]
Proceed exactly as in the proof of Theorem \ref{thm:SimpleZ}, using Proposition \ref{prop:HausdorffFiniteness} in place of  Proposition \ref{prop:SimpleFiniteness}.
\end{proof}

Let us discuss the relevance of these results.
For $\mathrm C^*$-algebras $A$ as in Theorem \ref{thm:HausdorffZ}, the main result of \cite{UnitlessZ} says that the simple quotients of $A$ are all $\ZZ$-stable.
Moreover, Hirshberg, R\o rdam, and Winter showed in \cite[Theorem 4.6]{HirshbergRordamWinter} that if the primitive ideal space of $A$ has finite covering dimension and all its simple quotients are $\ZZ$-stable, then $A$ is $\ZZ$-stable; thus, Theorem \ref{thm:HausdorffZ} only says something new in the case that the primitive ideal space of $A$ is infinite dimensional.
Examples of Hirshberg-R\o rdam-Winter and of Dadarlat show a variety of possibilities for $\mathrm C^*$-algebras with infinite-dimensional, Hausdorff primitive ideal space and $\ZZ$-stable  simple quotients \cite[Examples 4.7 and 4.8]{HirshbergRordamWinter}, \cite[Section 3]{Dadarlat:FiberwiseKK}.
Our result, combined with results of the second-named author and Winter in \cite{TW:Zdr}, neatly characterizes $\ZZ$-stability for such $\mathrm C^*$-algebras:

\begin{corollary}
Let $A$ be a finite $C^*$-algebra with Hausdorff primitive ideal space and no type I quotients.
Then the following are equivalent.
\begin{enumerate}
\item $A$ is $\ZZ$-stable, and there is a finite bound on the decomposition rank of the simple quotients of $A$;
\item $A$ has finite decomposition rank.
\end{enumerate}
\end{corollary}

\begin{proof}
(i) $\Rightarrow$ (ii) follows from \cite[Theorem 4.1 and Lemma 6.1]{TW:Zdr}.
(ii) $\Rightarrow$ (i) follows from Corollary \ref{cor:HausdorffZ}.
\end{proof}

The equivalence of the following two conditions, under the hypothesis of the above corollary, would follow from conjecture (C1):
\begin{enumerate}
\item[(i')] $A$ is $\ZZ$-stable, and there is a finite bound on the nuclear dimension of the simple quotients of $A$;
\item[(ii')] $A$ has finite nuclear dimension.
\end{enumerate}
(i') $\Rightarrow$ (ii') follows from \cite[Theorem 4.1 and Lemma 6.1]{TW:Zdr}.
In the case that every simple quotient is infinite, an implication similar to (but stronger than) (ii') $\Rightarrow$ (i') has been considered by Blanchard, Kirchberg, and R\o rdam in \cite{KirchbergRordam:pi2, BlanchardKirchberg:pi}.
In particular, using results of \cite[Theorem 8.6]{KirchbergRordam:pi2} and \cite[Theorem 5.8]{BlanchardKirchberg:pi}, it suffices to show in this case that the $\mathrm C^*$-algebra has $0$-comparison.

More generally, note that if a $\mathrm C^*$-algebra $A$ has Hausdorff primitive ideal space $X$, then the set of points $x \in X$ corresponding to infinite simple quotients forms an open set.
This is because: a $\mathrm C^*$-algebra is infinite if and only if it contains a partial isometry $v$ such that $v^*v < vv^*$, and this is a stable relation.
Therefore, $A$ is an extension of the two cases (all simple quotients are infinite, and no simple quotients are infinite), and thus, the general case reduces to the case that every quotient is infinite.

\newcommand{\cstar}{$\mathrm C^*$}

\end{document}